\documentclass[11pt,a4paper,leqno]{amsart}

\usepackage[latin1]{inputenc}
\usepackage[T1]{fontenc}
\usepackage{amsfonts}
\usepackage{amsmath}
\usepackage{amssymb}
\usepackage{eurosym}
\usepackage{mathrsfs}
\usepackage{palatino}
\usepackage{color}
\usepackage{esint}
\usepackage{url}
\usepackage{verbatim}

\usepackage{enumerate}

\usepackage[pagebackref,hypertexnames=false, colorlinks, citecolor=blue, linkcolor=blue, urlcolor=red]{hyperref}
\usepackage{amsrefs}

\newcommand{\R}{\mathbb{R}}
\newcommand{\C}{\mathbb{C}}

\newcommand{\N}{\mathbb{N}}

\newcommand{\Z}{\mathbb{Z}}
\newcommand{\E}{\mathbb{E}}

\newcommand{\bla}{\big \langle}
\newcommand{\bra}{\big \rangle}

\numberwithin{equation}{section}

\newcommand{\ud}{\,\mathrm{d}}

\newcommand{\esssup}{\operatornamewithlimits{ess\,sup}}

\newcommand{\abs}[1]{|#1|}

\newcommand{\Norm}[2]{\|#1\|_{#2}}

\newcommand{\ave}[1]{\langle #1\rangle}
\newcommand{\Ave}[1]{\Big\langle #1\Big\rangle}

\newcommand{\BMO}{\operatorname{BMO}}
\newcommand{\bmo}{\operatorname{bmo}}

\newcommand{\loc}{\operatorname{loc}}

\renewcommand{\Re}{\operatorname{Re}}

\newcommand{\ch}{\operatorname{ch}}

\newcommand{\calD}{\mathcal{D}}

\newcommand{\wt}[1]{{\widetilde{#1}}}

\swapnumbers
\theoremstyle{plain}
\newtheorem{thm}[equation]{Theorem}
\newtheorem{lem}[equation]{Lemma}
\newtheorem{prop}[equation]{Proposition}

\theoremstyle{definition}
\newtheorem{defn}[equation]{Definition}

\theoremstyle{remark}
\newtheorem{rem}[equation]{Remark}

\title{Two-weight inequalities for multilinear commutators in product spaces}

\author{Emil Airta}
\author{Kangwei Li}
\author{Henri Martikainen}

\address[E.A.]{Department of Mathematics and Statistics, University of Helsinki, P.O.B. 68, FI-00014 University of Helsinki, Finland}
\email{emil.airta@helsinki.fi}

\address[K.L.]{Center for Applied Mathematics, Tianjin University, Weijin Road 92, 300072 Tianjin, China}
\email{kli@tju.edu.cn}

\address[H.M.]{Department of Mathematics and Statistics, Washington University in St. Louis, 1 Brookings Drive, St. Louis, MO 63130, USA}
\email{henri@wustl.edu}

\makeatletter
\@namedef{subjclassname@2010}{%
  \textup{2010} Mathematics Subject Classification}
\makeatother


\thanks{E. A. was supported by the Emil Aaltonen Foundation and by the Academy of Finland through the grant 327271.}
\thanks{K. L. was supported by the National Natural Science Foundation of China through project number 12001400.}
\thanks{While at the University of Helsinki H. M. was supported by the Academy of Finland through project numbers 294840 and 327271, and by the three-year research grant 75160010 of the University of Helsinki. }
\begin{document}

\allowdisplaybreaks

\begin{abstract}
This note is devoted to establishing two-weight estimates for commutators of singular integrals. We combine multilinearity with product spaces.
A new type of two-weight extrapolation result is used to yield the quasi-Banach range of estimates.
\end{abstract}

\maketitle

\section{Introduction}
Commutators have the general form $[b,T]\colon f \mapsto bTf - T(bf)$. Here $T$ is a singular integral operator
$$
Tf(x) = \int_{\R^d} K(x,y)f(y)\ud y.
$$
Well-known examples include the Hilbert transform $H$ in dimension $d = 1$, which has the kernel $K(x,y) = \frac{1}{x-y}$, and the Riesz transforms $R_j$ in dimensions $d \ge 2$, which have
the kernel $K_j(x,y) = \frac{x_j-y_j}{|x-y|^{d+1}}$, $j=1,\ldots, d$.

Our work revolves around the Coifman--Rochberg--Weiss \cite{CRW} result, where the two-sided estimate
$$
\|b\|_{\BMO} \lesssim \|[b,T]\|_{L^p(\R^d) \to L^p(\R^d)} \lesssim \|b\|_{\BMO}, \qquad p \in (1,\infty),
$$
was proved for a class of non-degenerate singular integrals $T$ on $\R^d$. Here $\BMO$ stands for functions of bounded mean oscillation:
$$
\Norm{b}{\BMO} :=\sup_I \fint_I  \abs{b-\ave{b}_I},
$$
where the supremum is over all cubes $I \subset \R^d$ and $\ave{b}_I = \fint_I b := \frac{1}{|I|} \int_I b$.
The corresponding two-weight problem concerns
estimates from $L^p(\mu)$ to $L^p(\lambda)$ for two different weights $\mu, \lambda$ and has recently attracted
interest after the work by Holmes--Lacey--Wick \cite{HPW}. See also e.g. \cites{LOR1,LOR2,HyCom}.

In this note we establish that two-weight estimates for commutators can be proved under the joint difficulty of multilinearity and product spaces. Both have been
considered separately before: see e.g.  \cites{EA, ALMV, HPW, LMV:Bloom, LMV:Bloom2} for the multi-parameter work, and \cite{Kunwar2018} and \cite{Li} for the multilinear
work. The recent satisfactory multilinear result of \cite{Li} is based on sparse domination and the approach cannot be used in our setting -- this is due to the product space
nature of the problem.
For given exponents $1 < p_1, \ldots, p_n \le \infty$ and $1/p = \sum_i 1/p_i> 0$, a natural form of a weighted estimate in the $n$-variable context has the form
$$
\Big \|T(f_1, \ldots, f_n)\prod_{i=1}^n w_i \Big \|_{L^p} \lesssim \prod_{i=1}^n \|f_i w_i\|_{L^{p_i}}.
$$
The key thing is to only impose a \emph{joint} condition on the tuple of weights $\vec w = (w_1, \ldots, w_n) \in A_{\vec p}$ rather
than to assume individual conditions $w_i^{p_i} \in A_{p_i}$.
See Lerner, Ombrosi, P\'erez, Torres and Trujillo-Gonz\'alez \cite{LOPTT} and for multi-parameter versions \cite{LMV:gen}.
Naturally, this interplay is trickier still in our two-weight setting.

Our result is the following.
\begin{thm}\label{thm:main}
Let $T$ be an $n$-linear bi-parameter Calder\'on-Zygmund operator. 
Let $\vec p = (p_1, \ldots, p_n)$ with $1 < p_i \le \infty$ and $$\frac{1}{p} = \sum_{i=1}^n \frac{1}{p_i} > 0.$$
With a fixed $j \in \{1, \ldots, n\}$
let $(w_1, \ldots, w_n)$  and $(w_1,\ldots, \lambda_j, \ldots, w_n)$ be two tuples of weights in the genuinely multilinear bi-parameter weight class $A_{\vec p}$ and define the associated Bloom weight $\nu = w_j \lambda_j^{-1}$. If we have
$b \in \bmo(\nu)$ and $\nu \in A_\infty$, then
$$
\|[b,T]_j(f_1, \ldots, f_n) \nu^{-1}w \|_{L^p} \lesssim \| b\|_{\bmo(\nu)} \prod_{i=1}^n \|f_i w_i\|_{L^{p_i}},\quad w = \prod_{i=1}^n w_i.
$$
The corresponding lower bound holds if $T$ is suitably non-degenerate.
\end{thm}
For the exact definitions see the main text.

Extrapolation methods are important in our current work -- they are used to yield the quasi-Banach range $p < 1$.
The extrapolation theorem of Rubio de Francia says that if $\|g\|_{L^{p_0}(w)} \lesssim \|f\|_{L^{p_0}(w)}$ for some $p_0 \in (1,\infty)$ and all $w \in A_{p_0}$, then $\|g\|_{L^{p}(w)} \lesssim \|f\|_{L^{p}(w)}$ for all $p \in (1,\infty)$ and all $w \in A_{p}$.
In \cite{GM} (see also \cite{DU}) a multivariable analogue was developed in the setting $w_i^{p_i} \in A_{p_i}$, $i = 1, \ldots, n$.
Very recently, in \cites{LMO, LMMOV, Nieraeth} it was shown that also the genuinely multilinear weighted estimates can be extrapolated. We prove
a suitable two-weight adaptation that can be used in our current work.
\begin{thm}\label{thm:extrapo}
Let $(f,f_1,\ldots, f_n)$ be a tuple of non-negative functions. Let $1\le p_i\le \infty$, $1\le i \le n,$  $ \frac{1}{p}=\sum_{i=1}^n\frac{ 1}{p_i},$ and $j \in \{1,\ldots,n\}.$ Assume that for all $(w_1, \cdots, w_n),$ $ (w_1,\ldots, \lambda_j, \ldots, w_n)\in A_{\vec p}$ with $w_j\lambda_j^{-1}\in A_\infty$, there holds that 
\[
\Big\|f \lambda_j \prod_{\substack{i=1\\i \neq j}}^n w_i \Big \|_{L^p}\lesssim \prod_{i=1}^n \|f_iw_i\|_{L^{p_i}}.
\]Then for all $(w_1, \cdots, w_n), (w_1, \ldots,\lambda_j,\cdots, w_n)\in A_{\vec q}$ with $w_j\lambda_j^{-1}\in A_\infty$ and $1<q_i \le \infty, i\neq j$, $1/q=1/{p_j}+\sum_{\substack{i=1\\i\neq j}}^n 1/{q_i}>0$, there holds that 
\[
\Big\|f\lambda_j \prod_{\substack{i=1\\i \neq j}}^n w_i\Big\|_{L^q}\lesssim \|f_jw_j\|_{L^{p_j}}\prod_{\substack{i=1\\i \neq j}}^n \|f_iw_i\|_{L^{q_i}}.
\]
\end{thm}

\section{Preliminaries}\label{sec:def}

Throughout this paper, $A\lesssim B$ means that $A\le CB$ with some constant $C$ that we deem unimportant to track at that point.
We write $A\sim B$ if $A\lesssim B\lesssim A$. Sometimes we e.g. write $A \lesssim_{\epsilon} B$ if we want to make the point that $A \le C(\epsilon) B$.

\subsection{Dyadic notation}
Given a dyadic grid $\calD$ in $\R^d$, $I \in \calD$ and $k \in \Z$, $k \ge 0$, we use the following notation:
\begin{enumerate}
\item $\ell(I)$ is the side length of $I$.
\item $I^{(k)} \in \calD$ is the $k$th parent of $I$, i.e., $I \subset I^{(k)}$ and $\ell(I^{(k)}) = 2^k \ell(I)$.
\item $\ch(I)$ is the collection of the children of $I$, i.e., $\ch(I) = \{J \in \calD \colon J^{(1)} = I\}$.
\item $E_I f=\langle f \rangle_I 1_I$ is the averaging operator, where $\langle f \rangle_I = \fint_{I} f = \frac{1}{|I|} \int _I f$.
\item $\Delta_If$ is the martingale difference $\Delta_I f= \sum_{J \in \ch (I)} E_{J} f - E_{I} f$.
\item $\Delta_{I,k} f$ is the martingale difference block
$$
\Delta_{I,k} f=\sum_{\substack{J \in \calD \\ J^{(k)}=I}} \Delta_{J} f.
$$
\end{enumerate}

For an interval $J \subset \R$ we denote by $J_{l}$ and $J_{r}$ the left and right
halves of $J$, respectively. We define $h_{J}^0 = |J|^{-1/2}1_{J}$ and $h_{J}^1 = |J|^{-1/2}(1_{J_{l}} - 1_{J_{r}})$.
Let now $I = I_1 \times \cdots \times I_d \subset \R^d$ be a cube, and define the Haar function $h_I^{\eta}$, $\eta = (\eta_1, \ldots, \eta_d) \in \{0,1\}^d$, by setting
\begin{displaymath}
h_I^{\eta} = h_{I_1}^{\eta_1} \otimes \cdots \otimes h_{I_d}^{\eta_d}.
\end{displaymath}
If $\eta \ne 0$ the Haar function is cancellative: $\int h_I^{\eta} = 0$. We exploit notation by suppressing the presence of $\eta$, and write $h_I$ for some $h_I^{\eta}$, $\eta \ne 0$. Notice that for $I \in \calD$ we have $\Delta_I f = \langle f, h_I \rangle h_I$ (where the finite $\eta$ summation is suppressed), $\langle f, h_I\rangle := \int fh_I$.

\subsection{Multi-parameter notation}
We will be working on the bi-parameter product space $\R^d = \R^{d_1} \times \R^{d_2}$.
We denote a general dyadic grid in $\R^{d_i}$ by $\calD^i$. We denote cubes in $\calD^i$ by $I^i, J^i, K^i$, etc.
Thus, our dyadic rectangles take the forms $ I^1 \times I^2$, $J^1 \times J^2$, $K^1 \times K^2$ etc. We usually
denote the collection of dyadic rectangles by $\calD = \calD^1 \times \calD^2$.

If $A$ is an operator acting on $\R^{d_1}$, we can always let it act on the product space $\R^d$ by setting $A^1f(x) = A(f(\cdot, x_2))(x_1)$. Similarly, we use the notation $A^i f$ if $A$ is originally an operator acting on $\R^{d_i}$. Our basic multi-parameter dyadic operators -- martingale differences and averaging operators -- are obtained by simply chaining together relevant one-parameter operators. For instance, a bi-parameter martingale difference is
$$
\Delta_R f = \Delta_{I^1}^1 \Delta_{I^2}^2 f, \qquad R = I^1 \times I^2.
$$
When we integrate with respect to only one of the parameters we may e.g. write
\[
\langle f, h_{I^1} \rangle_1(x_2):=\int_{\R^{d_1}} f(x_1, x_2)h_{I^1}(x_1) \ud x_1
\]
or
$$
\langle f \rangle_{I^1, 1}(x_2) := \fint_{I^1} f(x_1, x_2) \ud x_1.
$$

\subsection{Adjoints}
Consider an $n$-linear operator $T$ on $\R^d = \R^{d_1} \times \R^{d_2}$. Let $f_i = f_i^1 \otimes f_i^2$, $i = 1, \ldots, n+1$.
We set up notation for the adjoints of $T$ in the bi-parameter situation. We let $T^{j*}$, $j \in \{0, \ldots, n\}$, denote the full adjoints, i.e.,
$T^{0*} = T$ and otherwise
$$
\langle T(f_1, \dots, f_n), f_{n+1} \rangle
= \langle T^{j*}(f_1, \dots, f_{j-1}, f_{n+1}, f_{j+1}, \dots, f_n), f_j \rangle.
$$
A subscript $1$ or $2$ denotes a partial adjoint in the given parameter -- for example, we define
$$
\langle T(f_1, \dots, f_n), f_{n+1} \rangle
= \langle T^{j*}_1(f_1, \dots, f_{j-1}, f_{n+1}^1 \otimes f_j^2, f_{j+1}, \dots, f_n), f_j^1 \otimes f_{n+1}^2 \rangle.
$$
Finally, we can take partial adjoints with respect to different parameters in different slots also -- in that case we denote the adjoint by $T^{j_1*, j_2*}_{1,2}$. It simply interchanges
the functions $f_{j_1}^1$ and $f_{n+1}^1$ and the functions $f_{j_2}^2$ and $f_{n+1}^2$. Of course, we e.g. have $T^{j^*, j^*}_{1,2} = T^{j*}$ and $T^{0*, j^*}_{1,2} = T^{j*}_{2}$,
so everything can be obtained, if desired, with the most general notation $T^{j_1*, j_2*}_{1,2}$.
In any case, there are $(n+1)^2$ adjoints (including $T$ itself). Similarly, the bi-parameter dyadic model operators that we later define always have $(n+1)^2$ different forms.

\subsection{Multilinear bi-parameter weights}\label{sec:weights}

A weight $w(x_1, x_2)$ (i.e. a locally integrable a.e. positive function) belongs to the bi-parameter weight class $A_p = A_p(\R^{d_1} \times \R^{d_2})$, $1 < p < \infty$, if
$$
[w]_{A_p} := \sup_{R}\, \ave{w}_R \ave{w^{1-p'}}^{p-1}_R
= \sup_{R}\, \ave{w}_R \ave{w^{-\frac{1}{p-1}}}^{p-1}_R
< \infty,
$$
where the supremum is taken over rectangles $R$ -- that is, over $R = I^1 \times I^2$ where $I^i \subset \R^{d_i}$ is a cube. In contrast to the one-parameter definition,  we take supremum over rectangles instead of cubes.

We have
\begin{equation}\label{eq:eq28}
[w]_{A_p(\R^{d_1} \times \R^{d_2})} < \infty \textup { iff } \max\big( \esssup_{x_1 \in \R^{d_1}} \,[w(x_1, \cdot)]_{A_p(\R^{d_2})}, \esssup_{x_2 \in \R^{d_2}}\, [w(\cdot, x_2)]_{A_p(\R^{d_1})} \big) < \infty,
\end{equation}
and that $$
\max\big( \esssup_{x_1 \in \R^{d_1}} \,[w(x_1, \cdot)]_{A_p(\R^{d_2})}, \esssup_{x_2 \in \R^{d_2}}\, [w(\cdot, x_2)]_{A_p(\R^{d_1})} \big) \le [w]_{A_p(\R^{d_1}\times \R^{d_2})},
$$
while the constant $[w]_{A_p}$ is dominated by the maximum to some power.
It is also useful that $\ave{w}_{I^2,2} \in A_p(\R^{d_1})$ uniformly on the cube $I^2 \subset \R^{d_2}$.
For basic bi-parameter weighted theory see e.g. \cite{HPW}.
We say $w\in A_\infty(\R^{d_1}\times \R^{d_2})$ if
\[
[w]_{A_\infty}:=\sup_R \, \ave{w}_R \exp\big( \ave{\log w^{-1}}_R \big)<\infty.
\]
It is well-known that
$$A_\infty(\R^{d_1}\times \R^{d_2})=\bigcup_{1<p<\infty}A_p(\R^{d_1}\times \R^{d_2}).$$
We also define
$$
[w]_{A_1} = \sup_R \, \ave{w}_R \esssup_R w^{-1}.
$$

The following multilinear reverse H\"older property is well-known -- for the history and a very short proof see e.g. \cite{Li}*{Lemma 2.5}. The proof in our bi-parameter setting
is the same.
\begin{lem}\label{lem:RH}
Let $u_i \in (0, \infty)$ and $w_i \in A_{\infty}$, $i = 1, \ldots, N$, be bi-parameter weights. Then for every rectangle $R$ we have
$$
\prod_{i=1}^N \ave{w_i}_R^{u_i} \lesssim \Big\langle \prod_{i=1}^N w_i^{u_i} \Big\rangle_R.
$$
\end{lem}

Next we define multilinear bi-parameter Muckenhoupt weights. Original one-parameter versions appeared in \cite{LOPTT}. Our definition in the bi-parameter case is the same as in \cite{LMV:gen}. 
\begin{defn}\label{defn:defn1}
Given $\vec p=(p_1, \ldots, p_n)$ with $1 \le p_1, \ldots, p_n \le \infty$ we say that
$\vec{w}=(w_1, \ldots, w_n)\in A_{\vec p} = A_{\vec p}(\R^{d_1} \times \R^{d_2})$, if
$$
0<w_i <\infty, \qquad i = 1, \ldots, n,
$$
almost everywhere and
$$
[\vec{w}]_{A_{\vec p}}
:=\sup_R \, \ave{w^p}_R^{\frac 1 p} \prod_{i=1}^n \ave{w_i^{-p_i'}}_R^{\frac 1{p_i'}} < \infty,
$$
where the supremum is over rectangles $R$,
$$
w := \prod_{i=1}^n w_i \qquad \textup{and} \qquad \frac 1 p = \sum_{i=1}^n \frac 1 {p_i}.
$$
If $p_i = 1$ we interpret $\ave{w_i^{-p_i'}}_R^{\frac 1{p_i'}}$ as
$\esssup_R w_i^{-1}$, and if $p = \infty$ we interpret
$\ave{w^p}_R^{\frac 1 p}$ as $\esssup_R w$.
\end{defn}

Conveniently, we can characterize the class $A_{\vec p}$ using the standard $A_p$ class. The lemma is proven in \cite{LOPTT} and the bi-parameter analog of the same proof is recorded in \cite{LMV:gen}.
\begin{lem}\label{lem:lem1}
Let $\vec p=(p_1, \ldots, p_n)$ with $1 \le p_1, \ldots, p_n \le \infty$, $1/p = \sum_{i=1}^n 1/p_i \ge 0$,
$\vec{w}=(w_1, \ldots, w_n)$ and $w = \prod_{i=1}^n w_i$.
We have
$$
[w_i^{-p_i'}]_{A_{np_i'}} \le [\vec{w}]_{A_{\vec p}}^{p_i'}, \qquad i = 1, \ldots, n,
$$
and
$$
[w^p]_{A_{np}} \le [\vec{w}]_{A_{\vec p}}^{p}.
$$
In the case $p_i = 1$ the estimate is interpreted as $[w_i^{\frac 1n}]_{A_1} \le [\vec{w}]_{A_{\vec p}}^{1/n}$, and in the case $p=\infty$
we have $[w^{-\frac{1}{n}}]_{A_1} \le [\vec{w}]_{A_{\vec p}}^{1/n}$.

Conversely, we have
$$
[\vec{w}]_{A_{\vec p}} \le [w^p]_{A_{np}}^{\frac{1}{p}} \prod_{i=1}^n [w_i^{-p_i'}]_{A_{np_i'}}^{\frac{1}{p_i'}}.
$$
\end{lem}

Most of the proofs are duality based and this makes the following lemma relevant.
\begin{lem}[\cite{LMS}*{Lemma 3.1}]\label{lem:lem7}
Let $\vec p=(p_1, \ldots, p_n)$ with $1 < p_1, \ldots, p_n < \infty$ and $\frac 1 p = \sum_{i=1}^n \frac 1 {p_i} \in (0,1)$.
Let $\vec{w}=(w_1, \ldots, w_n)\in A_{\vec p}$ with $w = \prod_{i=1}^n w_i$ and define
\begin{align*}
\vec w^{\, i} &= (w_1, \ldots, w_{i-1}, w^{-1}, w_{i+1}, \ldots, w_n), \\
\vec p^{\,i} &= (p_1, \ldots, p_{i-1}, p', p_{i+1}, \ldots, p_n).
\end{align*}
Then we have
$$
[\vec{w}^{\,i}]_{A_{\vec p^{\, i}}} =
[\vec{w}]_{A_{\vec p}}.
$$
\end{lem}

In the main theorems of this paper we will be using the multilinear bi-parameter weights
$$(w_1,\ldots,w_n), (\lambda_1,w_2,\ldots,w_n) \in A_{(p_1,\ldots,p_n)}, \quad \text{and}\quad \nu := \lambda_1^{-1} w_1 \in A_\infty,$$ where $1 \le p_1, \ldots, p_n \le \infty$, $1/p = \sum_{i=1}^n 1/p_i > 0.$ Throughout this paper, we will be using notation $\sigma_i = w_{i}^{-p_i'}, \sigma_{n+1} = (\nu^{-1} w)^p,$ and $\eta_1= \lambda_1^{-p_1'}$ as they will appear regularly.

The assumption that $\nu \in A_\infty$ is necessary as it is not implied by the other two assumptions, see a counter-example in \cite{Li}.

However, instead of the two separate conditions 
$$
(w_1,\ldots,w_n) \in A_{(p_1,\ldots,p_n)} \quad\text{and} \quad(\lambda_1,w_2,\ldots,w_n) \in A_{(p_1,\ldots,p_n)},
$$ if we assume only that  $(w_1,\ldots,w_n,\nu w^{-1}) \in A_{(p_1,\ldots,p_n,p')},$ where $\nu = \lambda_1^{-1} w_1$ and $w = \prod_{i = 1}^n w_i,$ that is 
$$
\sup_R \prod_{i=1}^n \ave{w_i^{-p_i'}}_R^{\frac{1}{p_i'}} \ave{\nu^{-p} w^{p}}_R^{\frac 1p} \ave{\nu}_R < \infty,
$$
we would automatically get that 
$$
\prod_{i=1}^n w_i  \cdot \nu w^{-1} = \nu \in A_{n+1} \subset A_\infty
$$
by Lemma \ref{lem:lem1}.

Yet, it is unlikely that this assumption is enough for the  boundedness of the commutator as conjectured for the linear case in \cite{LORR:twoweight}. Although, we will show below that this assumption is enough for the boundedness of Bloom type paraproducts in the Banach range and also sufficient to conclude the lower bound of the commutator.

On the other hand, the joint assumption for the weights is very natural for the two-weight commutator estimates since 
the assumption $(w_1,\ldots,w_n,\nu w^{-1}) \in A_{(p_1,\ldots, p_n,p')}$ is implied by the two separate 
multilinear weight conditions and $\nu \in A_\infty.$ 

This is easy to verify. Let $\sum_{i=1}^n \frac{1}{p_i} =: \frac1p > 1$ and assume that $(w_1,\ldots,w_n), (\lambda_1,w_2,\ldots,w_n) \in A_{(p_1,\ldots,p_n)},$ and $\nu := \lambda_1^{-1} w_1 \in A_\infty.$ 
\begin{align*}
\prod_{i=1}^n \ave{w_i^{-p_i'}}_R^{\frac{1}{p_i'}} \ave{(\nu w^{-1})^{-p}}_R^{\frac 1p} \ave{\nu}_R  &=\prod_{i=1}^n \ave{w_i^{-p_i'}}_R^{\frac{1}{p_i'}} \ave{\lambda_1^p\prod_{i=2}^n w_i^p}_R^{\frac 1p} \ave{(\lambda_1^{-p_1'})^{\frac{1}{p_1'}}\prod_{i=2}^n (w_i^{-p_i'})^{\frac{1}{p_i'}} (w^{p})^{\frac 1p}}_R \\
&\stackrel{(*)}{\lesssim}\prod_{i=1}^n \ave{w_i^{-p_i'}}_R^{\frac{1}{p_i'}} \ave{\lambda_1^p\prod_{i=2}^n w_i^p}_R^{\frac 1p} \ave{\lambda_1^{-p_1'}}_R^{\frac{1}{p_1'}}\prod_{i=2}^n\ave{ w_i^{-p_i'}}_R^{\frac{1}{p_i'}} \ave{ w^{p}}_R^{\frac 1p} \\
&\leq [\vec w]_{A_{\vec p}}[(\lambda_1,w_2, \ldots,w_n)]_{A_{\vec{p}}},
\end{align*}
where in the step $(*)$ we apply \cite{Li}*{Lemma 2.9} for $\nu \in A_\infty.$

Motivated by the above discussion we give the following definition, where $p'$ does not appear hence $p>1$ is not needed.
\begin{defn}\label{def:Astar}
Given $\vec p=(p_1, \ldots, p_n)$ with $1 \le p_1, \ldots, p_n \le \infty,$ we say that
$\vec{w}=(w_1, \ldots, w_n, w_{n+1})\in A_{\vec p}^{*} = A_{\vec p}^*(\R^{d_1} \times \R^{d_2})$, if
$$
0<w_i <\infty, \qquad i = 1, \ldots, n+1,
$$
almost everywhere and
$$
[\vec{w}]_{A_{\vec p}^*}
:=\sup_R \,  \ave{w}_R \ave{w_{n+1}^{-p}}_R^{\frac 1 p}\prod_{i=1}^n \ave{w_i^{-p_i'}}_R^{\frac 1{p_i'}} < \infty,
$$
where the supremum is over rectangles $R$,
$$
w := \prod_{i=1}^{n+1} w_i \qquad \textup{and} \qquad \frac 1 p = \sum_{i=1}^n \frac 1 {p_i}.
$$
If $p_i = 1$ we interpret $\ave{w_i^{-p_i'}}_R^{\frac 1{p_i'}}$ as
$\esssup_R w_i^{-1}$, and if $p = \infty$ we interpret
$\ave{w^p}_R^{\frac 1 p}$ as $\esssup_R w$.
\end{defn}
Morally the difference is that with $A_{\vec p}^*$ we do not necessarily have  
$$
\prod_{i = 1}^{n}w_i^p \in A_\infty
$$
or
$
\lambda_j^{-p_j}\in A_\infty
$
 compared to assuming the two separate $A_{\vec{p}}$ and $\nu \in A_\infty$ but we are equipped with $\nu \in A_{n+1}.$

Furthermore, using this definition, we can write the following joint condition
$$
(w_1,\ldots,w_n,\nu w^{-1}) \in A_{(p_1,\ldots,p_n,p')} 
$$
as 
$(w_1,\ldots,w_n,\nu w^{-1}) \in A_{(p_1,\ldots,p_n)}^* = A_{\vec p}^*.$ 

%

\subsection*{$A_\infty$ extrapolation}

%

Besides of the extrapolation theorem proven in this paper, we also need to use the following $A_\infty$-extrapolation result of \cite{CUMP}.
\begin{lem}\label{lem:extInfty}
Let $(f,g)$ be a pair of non-negative functions. Suppose that there exists some $0<p_0<\infty$ such that for every $w\in A_\infty$ we have
\begin{equation}\label{eq:extrapol}
\int f^{p_0} w \lesssim \int g^{p_0} w.
\end{equation}
Then for all $0<p<\infty$ and $w\in A_\infty$ we have
$$
\int f^{p} w \lesssim \int g^{p} w.
$$
In addition, let $\{(f_i,g_i)\}_i$ be a sequence of  pairs of non-negative functions defined on $\R^d.$ Suppose that for some $0 < p_0 < \infty,$  $(f_i, g_i)$ satisfies inequality \eqref{eq:extrapol} for every $i.$
Then, for all $0 < p,q < \infty$ and $w \in A_\infty (\R^d)$ we have
$$
\Big\|\Big(\sum_i (f_i)^q\Big)^{\frac{1}{q}}\Big\|_{L^p(w)} \lesssim_{[w]_{A_\infty}} \Big\|\Big(\sum_i (g_i)^q\Big)^{\frac{1}{q}}\Big\|_{L^p(w)},
$$
where $\{(f_j,g_j)\}_j$ is a sequence of  pairs of non-negative functions defined on $\R^d.$
\end{lem}

\subsection{Maximal functions}
Let $\calD = \calD^1 \times \calD^2$
be a fixed lattice of dyadic rectangles and define
$$
M_{\calD}(f_1, \ldots, f_n) = \sup_{R \in \calD} \prod_{i=1}^n \ave{ |f_i| }_R 1_R.
$$
\begin{prop}\label{prop:prop2}
If $1 < p_1, \ldots, p_n \le \infty$ and $1/p = \sum_{i=1}^n 1/p_i$ we have
$$
\|M_{\calD}(f_1, \ldots, f_n)w \|_{L^p} \lesssim \prod_{i=1}^n \|f_i w_i\|_{L^{p_i}}
$$
for all multilinear bi-parameter weights $\vec w \in A_{\vec p}$.
\end{prop}
An efficient proof can be found in \cite{LMV:gen} (originally proved in \cite{GLTP}).

Also we often need the result of R. Fefferman \cite{RF3}. Proof also recorded in \cite{LMV:Bloom}*{Appendix B}. Denote
$\langle f \rangle_R^{\mu} := \frac{1}{\mu(R)} \int_R f\ud \mu$ and define
$$
M_{\calD}^{\mu} f = \sup_R 1_R \langle |f| \rangle_R^{\mu}.
$$
\begin{prop}\label{prop:prop1}
Let $\lambda \in A_p$, $p \in (1,\infty)$, be a bi-parameter weight. Then for all $s \in (1,\infty)$ we have
$$
\| M_{\calD}^{\lambda} f \|_{L^s(\lambda)} \lesssim [\lambda]_{A_p}^{1+1/s} \|f\|_{L^s(\lambda)}.
$$
\end{prop}

\subsection{Square functions}\label{sec:SF}
We begin with the classical (dyadic) square function in the bi-parameter framework.
Let $\calD = \calD^1 \times \calD^2$
be a fixed lattice of dyadic rectangles. We define the square functions
$$
S_{\calD} f = \Big( \sum_{R \in \calD} |\Delta_R f|^2 \Big)^{1/2}, \,\, S_{\calD^1}^1 f = \Big( \sum_{I^1 \in \calD^1} |\Delta_{I^1}^1 f|^2 \Big)^{1/2}
$$
and define $S_{\calD^2}^2 f$ analogously.

The lower bound estimate of the square function for $A_{\infty}$ weights is essential for many estimates later on. The fact that the key weights $w^p$ and $w_i^{-p_i'}$ are
at least in $A_{\infty}$ for the multilinear weights of Definition \ref{defn:defn1} allows us to use this lower bound estimate.
\begin{lem}\label{lem:lem3} It holds
$$
\|f\|_{L^p(w)} \lesssim \|S_{\calD^j}^j f\|_{L^p(w)} \lesssim \|S_{\calD} f\|_{L^p(w)}
$$
for all $p \in (0, \infty)$ and bi-parameter weights $w \in A_{\infty}$.
\end{lem}
The first inequality is the classical result found e.g. in \cite{Wi}*{Theorem 2.5} and the latter inequality  can be deduced using the $A_{\infty}$ extrapolation, Lemma \ref{lem:extInfty}.

Notice that by disjointness of supports we have, for example,
for all $k = (k_1, k_2) \in \{0,1,\ldots\}^2$ that
$$
S_{\calD} f = \Big( \sum_{K = K^1 \times K^2 \in \calD} |\Delta_{K,k} f|^2 \Big)^{1/2}, \qquad \Delta_{K,k} = \Delta_{K^1,k_1}^1 \Delta_{K^2, k_2}^2.
$$

Next, we take the definition of the $n$-linear square functions from \cite{LMV:gen}.
For $k= (k_1, k_2)$ we set
$$
A_1(f_1, \ldots, f_n) = A_{1,k}(f_1, \ldots, f_n)
= \Big( \sum_{K \in \calD} \langle | \Delta_{K,k} f_1 | \rangle_K ^2 \prod_{i=2}^n \langle |f_i| \rangle_K^2 1_K \Big)^{\frac{1}{2}}.
$$
In addition, we understand this so that $A_{1,k}$ can also take any one of the symmetric forms, where each $\Delta_{K^j, k_j}^j$ appearing in
$\Delta_{K,k} = \Delta_{K^1,k_1}^1 \Delta_{K^2, k_2}^2$ can alternatively be associated with any of the other functions $f_2, \ldots, f_n$. That is,
$A_{1,k}$ can, for example, also take the form
$$
A_{1,k}(f_1, \dots, f_n) =
\Big( \sum_{K \in \calD} \langle | \Delta^2_{K^2,k_2} f_1 | \rangle_K^2
\langle | \Delta^1_{K^1,k_1} f_2| \rangle_K^2 \prod_{i=3}^{n} \langle |f_i| \rangle_K^2 1_K \Big)^{\frac 12}.
$$
For $k = (k_1, k_2, k_3)$ we define
\begin{equation}\label{eq:eq11}
\begin{split}
&A_{2,k}(f_1, \ldots, f_n) \\
&= \Big( \sum_{K^2 \in \calD^2} \Big( \sum_{K^1 \in \calD^1}
\langle|\Delta^2_{K^2, k_1}f_1|\rangle_{K}\langle|\Delta^1_{K^1, k_2}f_2|\rangle_{K}
\langle|\Delta^1_{K^1, k_3}f_3|\rangle_{K} \prod_{i=4}^{n} \langle |f_i|\rangle_K1_{K}
\Big)^2\Big)^{ \frac 12},
\end{split}
\end{equation}
where we again understand this as a family of square functions. First, the appearing three martingale blocks can be
associated with different functions, too. Second, we can have the $K^1$ summation out and the $K^2$ summation in (we can interchange them), but then
we have two martingale blocks with $K^2$ and one martingale block with $K^1$.

Finally, for $k = (k_1, k_2, k_3, k_4)$ we define
$$
A_{3,k}(f_1, \ldots, f_n) = \sum_{K \in \calD} \langle | \Delta_{K,(k_1, k_2)} f_1| \rangle_K
\langle | \Delta_{K,(k_3, k_4)} f_2| \rangle_K \prod_{i=3}^n \langle |f_i| \rangle_K 1_K,
$$
where this is a family with two martingale blocks in each parameter, which can be moved around.
\begin{thm}[\cite{LMV:gen}*{Theorem 5.5.}]\label{thm:multilinSF}
If $1 < p_1, \ldots, p_n \le \infty$ and $\frac{1}{p} = \sum_{i=1}^n \frac{1}{p_i}> 0$ we have
$$
\|A_{j,k}(f_1, \ldots, f_n)w \|_{L^p} \lesssim \prod_{i=1}^n \|f_i w_i\|_{L^{p_i}}, \quad j=1,2,3,
$$
for all multilinear bi-parameter weights $\vec w \in A_{\vec p}$.
\end{thm}

Moreover, we need a certain linear estimate which appears regularly when dealing with the commutator estimates.
\begin{prop}[\cite{LMV:gen}*{Proposition 5.8.}]\label{prop:5.8}
For $u \in A_{\infty}$ and $p, s \in (1, \infty)$ we have
$$
\Big\| \Big[ \sum_m \Big( \sum_{K \in \calD} \langle |\Delta_{K,k} f_m| \rangle_K^2 \frac{1_K}{\langle u \rangle_K^2} \Big)^{\frac{s}{2}} \Big]^{\frac{1}{s}} u^{\frac{1}{p}}
\Big\|_{L^p}
\lesssim \Big\| \Big( \sum_m |f_m|^s \Big)^{\frac{1}{s}} u^{-\frac{1}{p'}} \Big\|_{L^p}.
$$
\end{prop}

\section{BMO spaces}
Let $\calD = \calD^1 \times \calD^2$ be a collection of dyadic rectangles on $\R^d = \R^{d_1} \times \R^{d_2}$.
For a function $b \in L^1_{\loc}$ and a bi-parameter weight $\nu \in A_{\infty}$ we define the
usual dyadic weighted little $\BMO$ norm of $b$ as follows:
$$
\|b\|_{\bmo(\nu)}:=\sup_{R \in \calD} \frac 1{\nu(R)}\int_R |b-\langle b\rangle_R|.
$$
In fact, the direct definition is not used that often and we will mostly invoke it through the following $H^1$-$\BMO$ type inequalities.
For $i = 1$ and $i=2$ we have
$$
|\langle b, f\rangle| \lesssim \| b\|_{\bmo(\nu)} \| S_{\calD^i}^i f \|_{L^1(\nu)} \lesssim \| b\|_{\bmo(\nu)} \| S_{\calD} f \|_{L^1(\nu)}.
$$
The first estimate follows from the one-parameter result \cite{Wu}, see e.g. \cite{HPW}. For the second inequality concerning square functions only see e.g. \cite{EA}*{Lemma 2.5}.

Often when a supremum is taken over rectangles we also have a formulation of the norm uniformly each parameter separately. We have
\begin{align}\label{eq:esssupBMO}
\|b\|_{\bmo(\nu)} \sim \max\big( \esssup_{x_1 \in \R^{d_1}} \|b(x_1,\cdot)\|_{\BMO(\nu(x_1, \cdot))}, \esssup_{x_2 \in \R^{d_2}}\|b(\cdot,x_2)\|_{\BMO(\nu(\cdot,x_2))}\big)
\end{align}
where $\|\cdot\|\BMO(\rho)$ is the standard one-parameter dyadic weighted BMO space. For proof see e.g. \cite{HPW}.

The following proposition gives an equivalent definition for the little $\BMO$ norm in Bloom type two-weight setting. The equivalent definition is needed for the proof of the lower bound of the commutator.
\begin{prop}\label{prop:bloomlittlebmo}
 Let $\nu,\sigma \in A_\infty.$ 
 If $\nu \sigma \in A_\infty,$
 then it holds
 $\bmo_\sigma(\nu) =  \bmo(\nu),$
 where  $$
 \bmo_\sigma(\nu) :=\{b \in L^1_{loc}\colon\sup_R \frac{1}{\nu\sigma(R)} \int_R |b- \ave{b}_R^\sigma| \sigma < \infty\}.
 $$
 \end{prop}
 The proof can be adapted from the one-parameter version (see, for example, \cite{Li}).  In our case, the sparse method poses no problems as it can be adapted to rectangles when the dyadic and sparse families inside of a rectangle $R$ are attained by iteratively bisecting the size of $R$. We omit the details.

We formulate the Muckenhoupt--Wheeden type estimates now.
\begin{lem}\label{lem:MW}
Let $a \in \BMO$ and $w \in A_\infty.$ It holds
$$
\sum_{I \in \calD} \ave{a,h_I} \ave{w}_I \varphi_I \lesssim \|a\|_{\BMO} \Big\|\Big(\sum_I \varphi_I^{2} \frac{1_I}{|I|}\Big)^\frac12\Big\|_{L^1(w)}.
$$
\end{lem}

In particular the above one is a special case of the two-weight version. We state this as a little bmo version.
\begin{lem}\label{lem:MWestimate}
Let $\sigma, \nu \in A_\infty.$
Assume that $b \in \bmo(\nu).$
Then we have
$$
\sum_{R = R^1 \times R^2} \ave{b, h_R} \ave{\sigma}_{R} \varphi_R \lesssim \|b\|_{\bmo(\nu)}\Big\|\Big(\sum_{R} \varphi_R^2 \frac{1_R}{|R|} \Big)^\frac12\Big\|_{L^1(\sigma \nu)}.
$$
Also, we have
$$
\sum_{R = R^1 \times R^2} \Ave{b, h_{R^1} \otimes \frac{1_{R^2}}{|R^2|}} \ave{\sigma}_{R} \varphi_R \lesssim \|b\|_{\bmo(\nu)}\Big\|\sum_{R^2}\Big(\sum_{R^1} \varphi_R^2 \frac{1_{R^1}}{|R^1|} \Big)^\frac12 \otimes \frac{1_{R^2}}{|R^2|}\Big\|_{L^1(\sigma \nu)}
$$
with a similar estimate when the cancellation is on the second parameter. 
\end{lem}


\begin{proof}
Let us consider the first estimate above and use the duality
\begin{align*}
\sum_{R = R^1 \times R^2} \ave{b, h_R} \ave{\sigma}_{R} \varphi_R \lesssim \|b\|_{\bmo(\nu)}\int \Big( \sum_{R} \varphi_R^2 \ave{\sigma}_R^2 \frac{1_R}{|R|}\Big)^\frac12 \nu.
\end{align*}
By the reverse H\"older property of $A_\infty$ weights, Lemma \ref{lem:RH}, we have
$$
\ave{\sigma}_{R}\ave{\nu}_R \lesssim \ave{\sigma \nu}_R.
$$
Hence, for all $R \in \calD$ we have
$$
\int \varphi_R \ave{\sigma}_R \frac{1_R}{|R|} \nu \lesssim \int \varphi_R \frac{1_R}{|R|} \sigma \nu.
$$

The second part of the extrapolation result, Lemma \ref{lem:extInfty}, yields that
$$
\int \Big( \sum_{R} \varphi_R^2 \ave{\sigma}_R^2 \frac{1_R}{|R|}\Big)^{\frac 12} \nu \lesssim \int \Big( \sum_{R} \varphi_R^2 \frac{1_R}{|R|}\Big)^\frac12 \sigma\nu
$$
as desired. 

For the second claim observe that, for example, we have
\begin{align*}
\sum_{R = R^1 \times R^2} \Ave{b, h_{R^1} \otimes \frac{1_{R^2}}{|R^2|}} \ave{\sigma}_{R} \varphi_R &= \int_{\R^{d_2}} \sum_{R^2} \sum_{R^1}\ave{b, h_{R^1}}  \ave{\sigma}_{R} \varphi_R \frac{1_{R^2}}{|R^2|} \\
&\lesssim \|b\|_{\bmo(\nu)} \int_{\R^d} \sum_{R^2} \Big(\sum_{R^1} \varphi_R^2 \ave{\sigma}_{R}^2 \frac{1_{R^1}}{|R^1|} \Big)^{\frac 12} \otimes \frac{1_{R^2}}{|R^2|} \nu,
\end{align*}
where we use the one-parameter duality for fixed variable on the second parameter. The proof is concluded as above.
\end{proof}

Using characterizations \eqref{eq:esssupBMO} and  \eqref{eq:eq28}, we have
\begin{lem}\label{lem:MWuniform}
Let $\sigma, \nu \in A_\infty.$
Assume that $b \in \bmo(\nu).$
For a fixed variable $x_1 \in \R^{d_1},$ we have
$$
\sum_{R^2} \ave{b_{x_1}, h_{R^2}} \ave{\sigma_{x_1}}_{R^2} \varphi_{R^2} \lesssim \|b\|_{\bmo(\nu)} \Big\|\Big(\sum_{R^2} \varphi_{R^2}^2 \frac{1_{R^2}}{|R^2|} \Big)^\frac12\Big\|_{L_{x_2}^1(\sigma_{x_1} \nu_{x_1})},
$$
where $g_{x_1}$ denotes the one parameter function $g(x_1, \cdot).$ We have a similar estimate for a fixed variable on $\R^{d_2}.$
\end{lem}
We omit the proof as it is analogous to the previous one.

\section{Multilinear bi-parameter singular integrals}
We call a function $\omega$ as a modulus of continuity if it is an increasing and subadditive function with $\omega(0) = 0$. 
A relevant quantity is
the modified Dini condition
\begin{equation*}
\|\omega\|_{\operatorname{Dini}_{\alpha}} := \int_0^1 \omega(t) \Big( 1 + \log \frac{1}{t} \Big)^{\alpha} \frac{dt}{t}, \qquad \alpha \ge 0
\end{equation*}
that appears in practise as follows
$$
\sum_{k=1}^\infty \omega(2^{-k})k^\alpha = \sum_{k=1}^\infty \frac{1}{\log 2} \int_{2^{-k}}^{2^{-k+1}} \omega(2^{-k})k^\alpha \frac{\ud t}{t}\lesssim \int_0^1 \omega(t) \Big( 1 + \log \frac{1}{t} \Big)^{\alpha} \frac{dt}{t}.
$$

\subsection{Bi-parameter SIOs}
We consider an $n$-linear operator $T$ on $\R^d = \R^{d_1}\times\R^{d_2}.$ 
Let $\omega_i$ be a modulus of continuity on $\R^{d_i}$. We define that $T$ is an $n$-linear bi-parameter $(\omega_1, \omega_2)$-SIO if it satisfies the full and partial kernel representations as defined below.

\subsubsection*{Full kernel representation}
Let  $f_i = f_i^1 \otimes f_i^2, i= 1,\ldots,n+1.$ For both $m \in \{1,2\}$ there exists $i_1, i_2 \in \{1, \ldots, n+1\}$ so that
$\operatorname{spt} f_{i_1}^m \cap \operatorname{spt} f_{i_2}^m = \emptyset$.
We demand that in this case we have the representation
$$
\langle T(f_1, \ldots, f_n), f_{n+1}\rangle = \int_{\R^{(n+1)d}}  K(x_{n+1},x_1, \dots, x_n)\prod_{i=1}^{n+1} f_i(x_i) \ud x,
$$
where
$$
K \colon \R^{(n+1)d} \setminus \{ (x_{n+1}, x_1, \ldots, x_{n}) \in \R^{(n+1)d}\colon x_1^1 = \cdots =  x_{n+1}^1 \textup{ or }  x_1^2 = \cdots =  x_{n+1}^2\} \to \C
$$
is a kernel satisfying a set of estimates which we specify next. 
The kernel $K$ is assumed to satisfy the size estimate
\begin{displaymath}
|K(x_{n+1},x_1, \dots, x_n)| \lesssim \prod_{m=1}^2 \frac{1}{\Big(\sum_{i=1}^{n} |x_{n+1}^m-x_i^m|\Big)^{d_mn}}.
\end{displaymath}

In addition, we require the continuity estimate, for example, we demand that
\begin{align*}
|K(x_{n+1}, x_1, \ldots, x_n)-&K(x_{n+1},x_1, \dots, x_{n-1}, (c^1,x^2_n))\\
&-K((x_{n+1}^1,c^2),x_1, \dots, x_n)+K((x_{n+1}^1,c^2),x_1, \dots, x_{n-1},  (c^1,x^2_n))| \\
&\qquad \lesssim \omega_1 \Big( \frac{|x_{n}^1-c^1| }{ \sum_{i=1}^{n} |x_{n+1}^1-x_i^1|} \Big) 
\frac{1}{\Big(\sum_{i=1}^{n} |x_{n+1}^1-x_i^1|\Big)^{d_1n}} \\
&\qquad\times
\omega_2 \Big( \frac{|x_{n+1}^2-c^2| }{ \sum_{i=1}^{n} |x_{n+1}^2-x_i^2|} \Big) 
\frac{1}{\Big(\sum_{i=1}^{n} |x_{n+1}^2-x_i^2|\Big)^{d_2n}}
\end{align*}
whenever $|x_n^1-c^1| \le 2^{-1} \max_{1 \le i \le n} |x_{n+1}^1-x_i^1|$
and $|x_{n+1}^2-c^2| \le 2^{-1} \max_{1 \le i \le n} |x_{n+1}^2-x_i^2|$.
Of course, we also require all the other natural symmetric estimates, where $c^1$ can be in any of the given $n+1$ slots and similarly for $c^2$. There
are, of course, $(n+1)^2$ different estimates.

Moreover, we expect to have the following mixed continuity and size estimates. For example, we demand that
\begin{align*}
|K(x_{n+1}&, x_1, \ldots, x_n)-K(x_{n+1},x_1, \dots, x_{n-1}, (c^1,x^2_n))| \\
& \lesssim \omega_1 \Big( \frac{|x_{n}^1-c^1| }{ \sum_{i=1}^{n} |x_{n+1}^1-x_i^1|} \Big) 
\frac{1}{\Big(\sum_{i=1}^{n} |x_{n+1}^1-x_i^1|\Big)^{d_1n}} \cdot  \frac{1}{\Big(\sum_{i=1}^{n} |x_{n+1}^2-x_i^2|\Big)^{d_2n}}
\end{align*}
whenever $|x_n^1-c^1| \le 2^{-1} \max_{1 \le i \le n} |x_{n+1}^1-x_i^1|$. Again, we also require all the other natural symmetric estimates.

\subsubsection*{Partial kernel representations}
Suppose now only that there exists $i_1, i_2 \in \{1, \ldots, n+1\}$ so that
$\operatorname{spt} f_{i_1}^1 \cap \operatorname{spt} f_{i_2}^1 = \emptyset$.
 Then we assume that
$$
\langle T(f_1, \ldots, f_n), f_{n+1}\rangle = \int_{\R^{(n+1)d_1}} K_{(f_i^2)}(x_{n+1}^1, x_1^1, \ldots, x_n^1) \prod_{i=1}^{n+1} f_i^1(x^1_i) \ud x^1,
$$
where $K_{(f_i^2)}$ is a one-parameter $\omega_1$-Calder\'on--Zygmund kernel with a constant depending on the fixed functions $f_1^2, \ldots, f_{n+1}^2$.
For example, this means that the size estimate takes the form
$$
|K_{(f_i^2)}(x_{n+1}^1, x_1^1, \ldots, x_n^1)| \le C(f_1^2, \ldots, f_{n+1}^2) \frac{1}{\Big(\sum_{i=1}^{n} |x_{n+1}^1-x_i^1|\Big)^{d_1n}}.
$$
The continuity estimates are analogous.

We assume the following $T1$ type control on the constant $C(f_1^2, \ldots, f_{n+1}^2)$. We have
\begin{equation}\label{eq:PKWBP}
C(1_{I^2}, \ldots, 1_{I^2}) \lesssim |I^2|
\end{equation}
and
$$
C(a_{I^2}, 1_{I^2}, \ldots, 1_{I^2}) + C(1_{I^2}, a_{I^2}, 1_{I^2}, \ldots, 1_{I^2}) + \cdots + C(1_{I^2}, \ldots, 1_{I^2}, a_{I^2}) \lesssim |I^2|
$$
for all cubes $I^2 \subset \R^{d_2}$
and all functions $a_{I^2}$ satisfying $a_{I^2} = 1_{I^2}a_{I^2}$, $|a_{I^2}| \le 1$ and $\int a_{I^2} = 0$.

Analogous partial kernel representation on the second parameter is assumed when $\operatorname{spt} f_{i_1}^2 \cap \operatorname{spt} f_{i_2}^2 = \emptyset$
for some $i_1, i_2$.

\subsection{Multilinear bi-parameter Calder\'on-Zygmund operators}
We say that $T$ satisfies the weak boundedness property if
\begin{equation}\label{eq:2ParWBP}
|\langle T(1_R, \ldots, 1_R), 1_R \rangle| \lesssim |R|
\end{equation}
for all rectangles $R = I^1 \times I^2 \subset \R^{d} = \R^{d_1} \times \R^{d_2}$.

An SIO $T$ satisfies the diagonal BMO assumption if the following holds. For all rectangles $R = I^1 \times I^2 \subset \R^{d} = \R^{d_1} \times \R^{d_2}$
and functions $a_{I^i}$ with $a_{I^i} = 1_{I^i}a_{I^i}$, $|a_{I^i}| \le 1$ and $\int a_{I^i} = 0$ we have
\begin{equation}\label{eq:DiagBMO}
|\langle T(a_{I^1} \otimes 1_{I^2}, 1_R, \ldots, 1_R), 1_R \rangle| + \cdots +  |\langle T(1_R, \ldots, 1_R), a_{I^1} \otimes 1_{I^2} \rangle| \lesssim |R|
\end{equation}
and
$$
|\langle T(1_{I^1} \otimes a_{I^2}, 1_R, \ldots, 1_R), 1_R \rangle| + \cdots +  |\langle T(1_R, \ldots, 1_R), 1_{I^1} \otimes a_{I^2} \rangle| \lesssim |R|.
$$

An SIO $T$ satisfies the product BMO assumption if it holds
$$S(1,\cdots,1) \in \BMO_{\textup{prod}}$$ for all the $(n+1)^2$ adjoints $S = T^{j_1*, j_2*}_{1,2}$.
This can be interpreted in the sense that
$$
\| S(1,\cdots, 1) \|_{\BMO_{\operatorname{prod}}} = \sup_{\calD = \calD^1 \times \calD^2} \sup_{\Omega} \Big(\frac{1}{|\Omega|} \sum_{ \substack{ R = I^1 \times I^2 \in \calD \\
R \subset \Omega}} |\langle S(1,\cdots, 1), h_R \rangle|^2 \Big)^{1/2} < \infty,
$$
where $h_R = h_{I^1} \otimes h_{I^2}$ and the supremum is over all dyadic grids $\calD^i$ on $\R^{d_i}$ and
open sets $\Omega \subset \R^d = \R^{d_1} \times \R^{d_2}$ with $0 < |\Omega| < \infty$, and the pairings
$\langle S(1,\cdots,1), h_R\rangle$ can be defined, in a natural way, using the kernel representations.

\begin{defn}\label{defn:CZO}
An $n$-linear  bi-parameter $(\omega_1, \omega_2)$-SIO $T$
satisfying the weak boundedness property, the diagonal BMO assumption and the product BMO assumption is called an $n$-linear bi-parameter
$(\omega_1, \omega_2)$-Calder\'on--Zygmund operator ($(\omega_1, \omega_2)$-CZO). 
\end{defn}

We simplify the study of above operators through the following representation theorem.
\begin{prop}\label{prop:Representation}
Suppose $T$ is an $n$-linear bi-parameter $(\omega_1, \omega_2)$-CZO.
Then we have
$$
\langle T(f_1,\ldots,f_n), f_{n+1} \rangle= C_T \E_{\sigma} \sum_{u = (u_1, u_2) \in \N^2}  \omega_1(2^{-u_1})\omega_2(2^{-u_2}) \langle U_{u, \sigma}(f_1,\ldots,f_n), f_{n+1} \rangle,
$$
where $C_T$ enjoys a linear bound with respect to the CZO quantities and
$U_{u, \sigma}$ denotes some $n$-linear bi-parameter dyadic operator (defined in the grid $\calD_{\sigma}$) with the following property. We have that $U_u = U_{u, \sigma}$ can be 
decomposed using the standard dyadic model operators as follows:
\begin{equation}\label{eq:eq10}
U_{u}
= C \sum_{i_1=0}^{u_1-1} \sum_{i_2=0}^{u_2-1} V_{i_1,i_2},
\end{equation}
where each $V = V_{i_1,i_2}$ is a dyadic model operator (a shift, a partial paraproduct or a full paraproduct)
of complexity $k^m_{j, V}$, $j \in \{1, \ldots, n+1\}$, $m \in \{1,2\}$,
satisfying 
$$
k^{m}_{j, V} \le u_m.
$$
\end{prop}
In above $\E_{\sigma}$ denotes the expectation over  a natural probability space $\Omega = \Omega_1 \times \Omega_2$, the details of which are not relevant for us here,
so that to each $\sigma = (\sigma_1, \sigma_2) \in \Omega$
we can associate a random collection of dyadic rectangles $\calD_{\sigma} = \calD_{\sigma_1} \times \calD_{\sigma_2}$. The proposition is a consequence of \cite{AMV}*{Theorem 5.35. and Lemma 5.12.}. 

It was proven in \cite{AMV} that the minimal regularity we require is that $\omega_i \in \text{Dini}_{\frac12}.$ For the optimal dependence the dyadic representation is in terms of certain modified model operators. The modified versions of the standard operators are much more difficult to handle and we are forced to rely on the lemma that these can be written as a sum of the standard ones. However,  as it is explained in \cite{AMV}, this will cause a loss in the kernel regularity. 
Yet another problem appears when dealing with the genuinely multilinear weights. Thus in some cases, we need to stick to the usual H\"older type kernel regularity $\omega_i(t) = t^{\alpha_i}.$ In the paper \cite{LMV:gen}, it was proven that the standard model operators are bounded with the weights on the genuinely multilinear weight class introduced earlier. We will move on to introducing the model operators and state the very recent results for these.

\subsection{Dyadic model operators}\label{sec:model}
All the operators in this section are defined in some fixed rectangles $\calD = \calD^1 \times \calD^2$. We do not emphasise this dependence in the notation. 
\subsection{Shifts}
Let $k=(k_1, \dots, k_{n+1})$, where $k_i = (k_i^1, k_i^2) \in \{0,1,\ldots\}^2$.
An $n$-linear bi-parameter shift $S_k$ takes the form
\begin{equation*}\label{eq:S2par}
\langle S_k(f_1, \ldots, f_n), f_{n+1}\rangle = \sum_{K} \sum_{\substack{R_1, \ldots, R_{n+1} \\ R_i^{(k_i)} = K }}
a_{K, (R_i)} \prod_{i=1}^{n+1} \langle f_i, \wt h_{R_i} \rangle.
\end{equation*}
Here $K, R_1, \ldots, R_{n+1} \in \calD = \calD^1 \times \calD^2$, $R_i = I_i^1 \times I_i^2$, $R_i^{(k_i)} := (I_i^1)^{(k_i^1)} \times (I_i^2)^{(k_i^2)}$ and
$\wt h_{R_i} = \wt h_{I_i^1} \otimes \wt h_{I_i^2}$. Here we assume that for $m \in \{1,2\}$
there exist two indices $i^m_0,i_1^m \in \{1, \ldots, n+1\}$, $i^m_0 \not =i^m_1$, so that $\wt h_{I_{i^m_0}^m}=h_{I_{i^m_0}^m}$, $\wt h_{I_{i^m_1}^m}=h_{I_{i^m_1}^m}$ and for the remaining indices $i \not \in \{i^m_0, i^m_1\}$ we have $\wt h_{I_i^m} \in \{h_{I_i^m}^0, h_{I_i^m}\}$.
Moreover, $a_{K,(R_i)} = a_{K, R_1, \ldots ,R_{n+1}}$ is a scalar satisfying the normalization
\begin{equation}\label{eq:Snorm2par}
|a_{K,(R_i)}| \le \frac{\prod_{i=1}^{n+1} |R_i|^{1/2}}{|K|^{n}}.
\end{equation}

\begin{thm}[\cite{LMV:gen}*{Theorem 6.2.}]\label{thm:shift}
Suppose $S_k$ is an $n$-linear bi-parameter shift, $1 < p_1, \ldots, p_n, \le \infty$ and $\frac{1}{p} = \sum_{i=1}^n \frac{1}{p_i}> 0$. Then we have
$$
\|S_k(f_1, \ldots, f_n)w \|_{L^p} \lesssim \prod_{i=1}^n \|f_i w_i\|_{L^{p_i}}
$$
for all multilinear bi-parameter weights $\vec w \in A_{\vec p}$. The implicit constant does not depend on $k$.
\end{thm}

\subsection{Partial paraproducts}
Let $k=(k_1, \dots, k_{n+1})$, where $k_i \in \{0,1,\ldots\}$.
An $n$-linear bi-parameter partial paraproduct $(S\pi)_k$ with the paraproduct component on $\R^{d_2}$ takes the form
\begin{equation}\label{eq:Spi}
\langle (S\pi)_k(f_1, \ldots, f_n), f_{n+1} \rangle =
\sum_{K = K^1 \times K^2} \sum_{\substack{ I^1_1, \ldots, I_{n+1}^1 \\ (I_i^1)^{(k_i)} = K^1}} a_{K, (I_i^1)} \prod_{i=1}^{n+1} \langle f_i, \wt h_{I_i^1} \otimes u_{i, K^2} \rangle,
\end{equation}
where the functions $\wt h_{I_i^1}$ and $u_{i, K^2}$ satisfy the following.
There are $i_0,i_1 \in \{1, \ldots, n+1\}$, $i_0 \not =i_1$, so that $\wt h_{I_{i_0}^1}=h_{I_{i_0}^1}$, $\wt h_{I_{i_1}^1}=h_{I_{i_1}^1}$ and for the remaining indices $i \not \in \{i_0, i_1\}$ we have $\wt h_{I_i^1} \in \{h_{I_i^1}^0, h_{I_i^1}\}$. There is $i_2 \in \{1, \ldots, n+1\}$ so that $u_{i_2, K^2} = h_{K^2}$ and for the remaining indices $i \ne i_2$ we have
$u_{i, K^2} = \frac{1_{K^2}}{|K^2|}$.
Moreover, the coefficients are assumed to satisfy
\begin{equation}\label{eq:PPNorma}
\| (a_{K, (I_i^1)})_{K^2} \|_{\BMO} = \sup_{K^2_0 \in \calD^2} \Big( \frac{1}{|K^2_0|} \sum_{K^2 \subset K^2_0} |a_{K, (I_i^1)}|^2 \Big)^{1/2}
\le \frac{\prod_{i=1}^{n+1} |I_i^1|^{\frac 12}}{|K^1|^{n}}.
\end{equation}
Of course, $(\pi S)_k$ is defined symmetrically.

\begin{thm}[\cite{LMV:gen}*{Theorem 6.7.}]\label{thm:partials}
Suppose $(S\pi)_k$ is an $n$-linear partial paraproduct, $1 < p_1, \ldots, p_n \le \infty$ and $\frac{1}{p} = \sum_{i=1}^n \frac{1}{p_i}> 0$. Then, for every
$0<\beta \le 1$ we have
$$
\|(S\pi)_k(f_1, \ldots, f_n)w \|_{L^p} \lesssim_\beta 2^{\max_j k_j \beta}\prod_{i=1}^n \|f_i w_i\|_{L^{p_i}}
$$
for all multilinear bi-parameter weights $\vec w \in A_{\vec p}$.
\end{thm}

\subsection{Full paraproducts}
An $n$-linear bi-parameter full paraproduct $\Pi$ takes the form
\begin{equation}\label{eq:pi2bar}
\langle \Pi(f_1, \ldots, f_n) , f_{n+1} \rangle = \sum_{K = K^1 \times K^2} a_{K} \prod_{i=1}^{n+1} \langle f_i, u_{i, K^1} \otimes u_{i, K^2} \rangle,
\end{equation}
where the functions $u_{i, K^1}$ and $u_{i, K^2}$ are like in \eqref{eq:Spi}.
The coefficients are assumed to satisfy
$$
\| (a_{K} ) \|_{\BMO_{\operatorname{prod}}} = \sup_{\Omega} \Big(\frac{1}{|\Omega|} \sum_{K\subset \Omega} |a_{K}|^2 \Big)^{1/2} \le 1,
$$
where the supremum is over open sets $\Omega \subset \R^d = \R^{d_1} \times \R^{d_2}$ with $0 < |\Omega| < \infty$.

\begin{thm}[\cite{LMV:gen}*{Theorem 6.21.}]
Suppose $\Pi$ is an $n$-linear bi-parameter full paraproduct, $1 < p_1, \ldots, p_n \le \infty$ and $1/p = \sum_{i=1}^n 1/p_i> 0$. Then we have
$$
\|\Pi(f_1, \ldots, f_n)w \|_{L^p} \lesssim \prod_{i=1}^n \|f_i w_i\|_{L^{p_i}}
$$
for all multilinear bi-parameter weights $\vec w \in A_{\vec p}$.
\end{thm}

In fact, the above theorem is a special case of the Bloom type inequality. The following operator and result have obvious extensions in the product BMO setting.
We consider an $n$-linear bi-parameter paraproduct
\begin{equation}
\langle \Pi_{b}(f_1, \ldots, f_n) , f_{n+1} \rangle = \sum_{K = K^1 \times K^2} \langle b, v_{0, K^1} \otimes v_{0, K^2} \rangle \prod_{i=1}^{n+1} \langle f_i, v_{i, K^1} \otimes v_{i, K^2} \rangle.
\end{equation}
Here we assume that for $m \in \{1,2\}$
there exist two indices $i^m_0,i_1^m \in \{0, \ldots, n+1\}$, $i^m_0 \not =i^m_1$, so that $v_{i^m_0, K^m}=h_{K^m}$, $v_{i^m_1, K^m}=h_{K^m}$ and for the remaining indices $i \not \in \{i^m_0, i^m_1\}$ we have $v_{i, K^m}= \frac{1_{K^m}}{|K^m|}$. Moreover, here we will assume that we at least have $0 \in \{i^1_0, i^1_1\}$ or $0 \in \{i^2_0, i^2_1\}$.

Later on, paraproducts will also appear as a result of standard expansions of products
$$
bf = \sum_{I^i \in \calD^i} \ave{b,h_{I^i}}_i\ave{f,h_{I^i}}_i \otimes h_{I^i}h_{I^i} + \sum_{I^i \in \calD^i} \ave{b,h_{I^i}}_i\ave{f}_{I^i,i} \otimes h_{I^i} + \sum_{I^i \in \calD^i} \ave{b}_{{I^i},i}\ave{f,h_{I^i}}_i \otimes h_{I^i}.
$$
In the first term, the worst case is if $h_{I^i}h_{I^i}$ is non-cancellative hence equals to $1_{I^i}/|I^i|.$ Often it is enough to consider the worst-case scenario.

We denote these expansions as $\Pi_{j_1,j_2}(b,f), (j_1,j_2) \in  \{1,2,3\}^2,$ where the indices dictates the from of the paraproduct. More specifically, in the above language of the multilinear paraproduct: if $j_m = 1$ then $i_0^m = 0$ and $i_1^m = 1,$
if $j_m = 2$ then $i_0^m = 0$ and $i_1^m = 2,$ and if $j_m = 3$ then $i_0^m = 1$ and $i_1^m = 2.$ In all of the cases the unmentioned slot do not have the cancellation. Hence, notice that when $j_1 = 3 = j_2$ we have no cancellation for the function $b$ meaning that it is not a paraproduct as such.

\begin{prop}\label{prop:bloomPara}
Let $\Pi_{b}$ be a paraproduct as described above. Fix $\vec p = (p_1, \ldots, p_n)$ so that $1 < p_i \le \infty$, define $\frac{1}{p} = \sum_{i=1}^n \frac{1}{p_i}$ and assume $1 < p < \infty.$
Let $(w_1, \ldots, w_n, \nu)$
be a tuple of weights. Assume that
$$
b \in \bmo(\nu) \qquad \textup{and} \qquad (w_1, \ldots, w_n, \nu w^{-1}) \in A^*_{\vec{p}}.
$$
Then we have
\begin{equation}\label{eq:bloomPara}
\|\Pi_{b}(f_1, \ldots, f_n) \nu^{-1}w \|_{L^p} \lesssim \| b\|_{\bmo(\nu)} \prod_{i=1}^n \|f_i w_i\|_{L^{p_i}}.
\end{equation}

Moreover, if $\lambda_j = w_j \nu^{-1}$ for some $j = 1,2,\ldots n$ such that 
$$(w_1,\dots,w_{j-1}, \lambda_j,w_{j+1}, \ldots, w_n),(w_1,\ldots,w_n) \in A_{\vec p},$$
then
\eqref{eq:bloomPara} holds for all $1<p_i\le \infty$ such that $p \in (n^{-1},\infty).$
\end{prop}
\begin{proof}
It suffices to show that
$$
|\langle \Pi_{b}(f_1, \ldots, f_n), f_{n+1} \rangle |\lesssim \| b\|_{\bmo(\nu)} \prod_{i=1}^n \|f_i w_i\|_{L^{p_i}} \cdot \|f_{n+1} \nu w^{-1}\|_{L^{p'}}.
$$
\textbf{Case I.} We have $0 \in \{i^1_0, i^1_1\}$ and $0 \in \{i^2_0, i^2_1\}$. We consider the concrete case
\begin{align*}
&\langle \Pi_{b}(f_1, \ldots, f_n) , f_{n+1} \rangle\\
 &= \sum_{K = K^1 \times K^2} \langle b, h_{K^1} \otimes h_{K^2} \rangle
\Big \langle f_1, h_{K^1} \otimes \frac{1_{K^2}}{|K^2|} \Big \rangle
\Big \langle f_2, \frac{1_{K^1}}{|K^1|} \otimes h_{K^2} \Big \rangle
\prod_{i=3}^{n+1} \langle f_i \rangle_{K}.
\end{align*}
We have
\begin{align*}
\langle \Pi_{b}(f_1, \ldots, f_n) , f_{n+1} \rangle &\lesssim \| b\|_{\bmo(\nu)} \Big\| \Big( \sum_{K \in \calD} \langle | \Delta^1_{K^1} f_1 | \rangle_K^2
\langle | \Delta^2_{K^2} f_2| \rangle_K^2 \prod_{i=3}^{n+1} \langle |f_i| \rangle_K^2 1_K \Big)^{\frac 12} \Big\|_{L^1(\nu)} \\
&\lesssim \| b\|_{\bmo(\nu)} \prod_{i=1}^n \|f_i w_i\|_{L^{p_i}} \cdot \|f_{n+1} \nu w^{-1}\|_{L^{p'}}.
\end{align*}
Here the first step used that $\nu \in A_{\infty}$ -- which follows as $(w_1, \ldots, w_n, \nu w^{-1}) \in A^*_{\vec p}$ --
and the estimate
$$
|\langle b, f\rangle| \lesssim \| b\|_{\bmo(\nu)} \| S_{\calD} f \|_{L^1(\nu)}.
$$
The second step used Theorem \ref{thm:multilinSF} together with the assumption $(w_1, \ldots, w_n, \nu w^{-1}) \in A^*_{\vec p}.$ 
 
\textbf{Case 2.} We have $0 \not \in \{i^1_0, i^1_1\}$ but $0 \in \{i^2_0, i^2_1\}$ (or the other way around).
We consider the concrete case
\begin{align*}
&\langle \Pi_{b}(f_1, \ldots, f_n) , f_{n+1} \rangle \\
& = \sum_{K} \Big \langle b, \frac{1_{K^1}}{|K^1|} \otimes h_{K^2} \Big \rangle
\Big \langle f_1, \frac{1_{K^1}}{|K^1|} \otimes h_{K^2} \Big \rangle
\Big \langle f_2, h_{K^1} \otimes \frac{1_{K^2}}{|K^2|} \Big \rangle
\Big \langle f_3, h_{K^1} \otimes \frac{1_{K^2}}{|K^2|} \Big \rangle
\prod_{i=4}^{n+1} \langle f_i \rangle_{K}.
\end{align*}
We have
\begin{align*}
&\langle A(f_1, \ldots, f_n) , f_{n+1} \rangle \\
&\lesssim \| b\|_{\bmo(\nu)} \Big\| \Big( \sum_{K^2} \Big( \sum_{K^1}
\langle|\Delta^2_{K^2}f_1|\rangle_{K}\langle|\Delta^1_{K^1}f_2|\rangle_{K}
\langle|\Delta^1_{K^1}f_3|\rangle_{K} \prod_{i=4}^{n} \langle |f_i|\rangle_K1_{K}
\Big)^2\Big)^{ \frac 12} \Big\|_{L^1(\nu)} \\
&\lesssim \| b\|_{\bmo(\nu)} \prod_{i=1}^n \|f_i w_i\|_{L^{p_i}} \cdot \|f_{n+1} \nu w^{-1}\|_{L^{p'}},
\end{align*}
where we used the estimate
$$
|\langle b, f\rangle| \lesssim \| b\|_{\bmo(\nu)} \| S_{\calD^2}^2 f \|_{L^1(\nu)}
$$
and Theorem \ref{thm:multilinSF}.

The second claim is obtained by using extrapolation, Theorem \ref{thm:extrapo}.
\end{proof}

Let $\lambda$ and $w$ be bi-parameter weights such that for some $1<p <\infty$ we have $\lambda^{-p'}, \\w^{-p'} \in A_\infty.$ Assume also that $\nu := \lambda^{-1}w \in A_\infty$ and $b \in \bmo(\nu).$
Then we have a weighted variant of the paraproduct operator
\begin{equation}\label{eq:pi2barweighted}
\langle \Pi_{b,\eta}f_1, f_{2} \rangle = \sum_{K = K^1 \times K^2} \langle b, h_{K} \rangle \langle f_1, h_{ K} \rangle \ave{f_2}_{K}^\eta,
\end{equation}
where $\eta = \lambda^{-p'}.$ 
\begin{prop}\label{prop:bloomParaEta1}
let $p \in (1,\infty).$ Let $\lambda$ and $w$ be bi-parameter weights such that $\lambda^{-p'}, \\w^{-p'} \in A_\infty.$ 
Let $\Pi_{b,\eta}$ be a  weighted paraproduct operator defined via \eqref{eq:pi2barweighted}, we have
$$
\|\Pi_{b,\eta}(f) \lambda \|_{L^p} \lesssim \| b\|_{\bmo(\nu)}\|f w\|_{L^{p}}.
$$
\end{prop}
\begin{proof}
The result follows from a variant of techniques seen in the proof of Proposition \ref{prop:bloomPara}. For example, by duality we have terms like \eqref{eq:pi2barweighted}. Introducing a weight averages $\ave{\sigma}_K \ave{\sigma}_K^{-1} =1,$ where $\sigma = w^{-p'},$ we can apply Lemma \ref{lem:MWestimate}. 
Hence, we get \begin{align*}
|\langle\Pi_{b,\eta}f_1, f_{2} \rangle| &\lesssim \|b\|_{\bmo(\nu)} \int \Big(\sum_{K} \frac{\langle f_1, h_{ K} \rangle^2}{\ave{\sigma}_K^2} (\ave{f_2}_{K}^\eta)^2 \frac{1_K}{|K|}\Big)^\frac12 \sigma \nu \\
&\leq \|b\|_{\bmo(\nu)}\int M_\calD^{\eta}f_2 \Big(\sum_{K} \frac{\langle f_1, h_{ K} \rangle^2}{\ave{\sigma}_K^2}\frac{1_K}{|K|} \Big)^2 \sigma \nu\\
&\leq \|b\|_{\bmo(\nu)}\| M_\calD^{\eta}f_2\|_{L^{p'}(\eta)} \Big\|\Big(\sum_{K} \frac{\langle f_1, h_{ K} \rangle^2}{\ave{\sigma}_K^2} \frac{1_K}{|K|}\Big)^2 \sigma^{\frac{1}{p}}\Big\|_{L^p} \\
&\lesssim \|b\|_{\bmo(\nu)}\|f_2 \lambda^{-1}\|_{L^{p'}}  \|f_1 w\|_{L^p}.
\end{align*}
\end{proof}

In the same setting as above we can have, for example, the following mixed type weighted paraproduct
\begin{equation*}
\langle \Pi_{b,\eta}f_1, f_{2} \rangle = \sum_{K = K^1 \times K^2} \Big\langle b, h_{K^1} \otimes \frac{1_{K^2}}{|K^2|} \Big\rangle \langle f_1, h_{ K} \rangle \ave{\ave{f_2,h_{K^2}}_2}_{K^1}^{\ave{\eta}_{K^2,2}}.
\end{equation*}
 Symmetrical definition when we have $\Ave{b, \frac{1_{K^1}}{|K^1|} \otimes h_{K^2}}$. We also consider the case  
 \[
 \langle \Pi_{b,\eta}f_1, f_{2} \rangle = \sum_{K = K^1 \times K^2} \Big\langle b, h_{K^1} \otimes h_{K^2}\Big\rangle \Ave{ f_1, h_{ K^1}\otimes \frac{1_{K^2}}{|K^2|}  } \ave{\ave{f_2,h_{K^2}}_2}_{K^1}^{\ave{\eta}_{K^2,2}}.
 \]

\begin{prop}\label{prop:bloomParaMixedEta}
For a weighted paraproduct operator $\Pi_{b,\eta}$ as described above, we have
$$
|\langle \Pi_{b,\eta}f_1, f_{2} \rangle| \lesssim \| b\|_{\bmo(\nu)}\|f w\|_{L^{p}} \|S_\calD^{i} f_2 \lambda^{-1}\|_{L^{p'}},
$$
where $i$ is either 1 or 2 depending on which parameter the cancellation is.
\end{prop}
\begin{proof}
Let us, for example, consider the paraproduct written above, where we have $\Ave{b,h_{K^1} \otimes \frac{1_{K^2}}{|K^2|}}.$ Similar to the previous proof, we use Lemma \ref{lem:MWestimate} but this time the second claim. Then the main difference to the previous proof is that we face e.g. 
$$
\Big\| \Big(\sum_{K^2} M^{\ave{\eta}_{K^2,2}}_{\calD^1} (\ave{f,h_{K^2}}_2)^2 \otimes \frac{1_{K^2}}{|K^2|} \Big)^{\frac 12}\Big\|_{L^{p'}(\eta)}.
$$
Nevertheless, the claim follows quite easily via an extrapolation trick (see \cite{LMV:gen}*{Lemma 9.2}), since for fixed $p'=2$ we have 
$$
\int_{\R^{d_1}} \Big[M^{\ave{\eta}_{K^2,2}}_{\calD^1} (\ave{f,h_{K^2}}_2)\Big]^2 \ave{\eta}_{K^2,2} \lesssim [\eta]_{A_\infty} \int_{\R^{d_1}} \ave{f,h_{K^2}}_2^2 \ave{\eta}_{K^2,2}.
$$
\end{proof}

For the references below, we state a lemma regarding the square functions of partial paraproducts. For the lemma, it is relevant in which slots the cancellation appears. The square function can be taken corresponding to the cancellation on the $(n+1)$-th slot. For example, if $(S\pi)_k$ is a form of partial paraproduct such that there is a cancellation on the $(n+1)$-th slot on the second parameter, then we have the  boundedness of the second parameter square function of this operator, namely $S_{\calD^2}(S\pi)_k.$  Similarly, $S_{\calD^1}(S\pi)_k$ and $S_{\calD}(S\pi)_k$ must have the corresponding cancellation to be bounded.
\begin{lem}\label{lem:SFpartial}
Let $U$ be a square function of partial paraproduct stated in above. Let $1< p_i \le \infty$ and $\frac 1 p = \sum_{i=1}^n \frac{1}{p_i} > 0.$ It holds
$$
\Big\|U(f_1,\ldots,f_n) w \Big\|_{L^{p}} \lesssim_{\beta} 2^{\max_j k_j \beta} \prod_{i=1}^n \|f_i w_i\|_{L^{p_i}},
$$
where $w = \prod_{i=1}^n w_i, (w_1,\ldots,w_n) \in A_{(p_1,\ldots,p_n)}.$
\end{lem}

\begin{proof}
The result follows almost identically to the proof of \cite{LMV:gen}*{Theorem 6.7.}.
We take the partial paraproduct of the form
\[
\Big(\sum_{K\in \calD} \Big(\sum_{(I_i^1)^{(k_i)}=K^1}a_{K, (I_i^1)} \prod_{i=1}^{n} \Ave{f_i,\wt h_{I_i^1} \otimes \frac{1_{K^2}}{|K^2|}} h_{I_{n+1}^1}^0\otimes h_{K^2}^0\Big)^2 \Big)^{\frac 12}.
\]
Using the dualisation trick in \cite{LMV:gen} for $p>1$, we choose a sequence of functions $(f_{n+1,K})_K\in L^{p'}(\ell^2)$ with norm $\|(f_{n+1,K})_K\|_{L^{p_1}(\ell^2)}\le 1 $, and we look at
\begin{align*}
&\bigg|\sum_{K\in \calD} \sum_{(I_i^1)^{(k_i)}=K^1}a_{K, (I_i^1)}\prod_{i=1}^{n} \Ave{f_i,\wt h_{I_i^1} \otimes \frac{1_{K^2}}{|K^2|}} \langle f_{n+1,K} w, h_{I_{n+1}^1}^0\otimes h_{K^2}^0\rangle\bigg|\\
&\le \sum_{K^1}\sum_{(I_i^1)^{(k_i)}=K^1 } \frac{\prod_{i=1}^{n+1}|I_i^1|^{\frac 12}}{|K^1|^2}\int_{\R^{d_2}}\Big( \sum_{K^2}\Big| A_{K^2}(\langle f_1,\wt h_{I_1^1}\rangle, \ldots, \langle f_{n+1,K} w, h_{I_{n+1}^1}^0\rangle)\Big|^2\frac{1_{K^2}}{|K^2|}\Big)^{\frac 12},
\end{align*} 
where
\[
A_{K^2}(g_1,\ldots, g_{n+1})=\Big(\prod_{i=1}^n\langle g_i\rangle_{K^2}\Big)\langle g_{n+1}, h_{K^2}^0\rangle.
\]
We write
\[
|I_i^1|^{-\frac 12}\Big\langle f_i, h_{I_i^1}^0\otimes \frac{1_{K^2}}{|K^2|}\Big\rangle= \langle f_i\rangle_{K}+ \sum_{\ell_3=0}^{k_i-1}\sum_{(L_i^1)^{(\ell_i)}=K^1}\Big\langle f_i, h_{L_i^1}\otimes \frac{1_{K^2}}{|K^2|}\Big\rangle \langle h_{L_i^1}\rangle_{I_i^1}
\] for $i \in \{1,2,\ldots,n\}$ whenever we have the non-cancellative Haar function, expect when complexity is zero.

We are reduced to bounding
\begin{equation}
\begin{split}
&\sum_{K^1}\sum_{(L_i^1)^{(\ell_i)}=K^1 } \frac{\prod_{i=1}^n|L_i^1|^{\frac 12}}{|K^1|^n}\\
&\times\int_{\R^{d_2}}\Big( \sum_{K^2}\Big| A_{K^2}(\langle f_1,\wt h_{L_1^1}\rangle,\ldots, \langle f_n, \wt h_{L_n^1}\rangle, \langle f_{n+1,K} w, h_{L_{n+1}^1}^0\rangle)\Big|^2\frac{1_{K^2}}{|K^2|}\Big)^{\frac 12},
\end{split}
\end{equation}
where $\wt h_{L_i^1} = h_{L_i^1}$ for at least one index $i$, and $\ell_{n+1}=k_{n+1}$. Moreover, if $\wt h_{L_i^1} = h_{L_i^1}^0,$ then we have complexity $\ell_i = 0.$

We consider an example to see how we can use the idea in \cite{LMV:gen} in this setting. The goal is to prove
\[
\| g\|_{L^1}\le \Big(\prod_{i=1}^n \|f_i w_i\|_{L^{p_i}} \Big)\| \widetilde f_{n+1} w^{-1} \|_{L^{p'}},
\]where
\[
\widetilde f_1:=\Big(\sum_{K} |f_{1,K} w|^2\Big)^{\frac 12}
\]and $g$ equals to
\begin{align*}
&\sum_{K^1}\sum_{(L_i^1)^{(\ell_i)}=K^1 } \frac{\prod_{i=1}^n|L_i^1|^{\frac 12}}{|K^1|^n} \frac{1_{K^1}}{|K^1|}\\
&\hspace{3cm}\times\Big( \sum_{K^2}\Big| A_{K^2}(\langle f_1,\wt h_{L_2^1}\rangle,\ldots, \langle f_n, \wt h_{L_n^1}\rangle, \langle f_{n+1,K} w, h_{L_{n+1}^1}^0\rangle)\Big|^2\frac{1_{K^2}}{|K^2|}\Big)^{\frac 12}.
\end{align*}
By extrapolation \cite{LMO},  we just need to prove that
\[
\| gv\|_{L^{\frac{2}{n+1}}}\le \prod_{i = 1}^n \|f_i v_i\|_{L^{2}} \| \widetilde f_{n+1}v_{n+1} \|_{L^{2}},\qquad (v_1,\ldots, v_{n+1})\in A_{(2,\ldots,2)}.
\]
Following the proof in \cite{LMV:gen}, everything will be the same except that for $\tilde f_{n+1}$, we need to control
\[
\Big\| \Big(\sum_{K^1}|F_{{n+1},K^1}|^2\Big)^{\frac 12} v_{n+1}^{-1}\Big\|_{L^2}= \Big\| \Big(\sum_{K^1}|F_{{n+1},K^1}|^2\Big)^{\frac 12} \gamma_{n+1}^{\frac 12}\Big\|_{L^2},
\]where
\[
F_{n+1,K^1}=1_{K^1}\sum_{(I_{n+1}^1)^{(k_{n+1})}=K^1}\frac{|I_{n+1}^1|^{\frac 12}}{|K^1|}\Big(\sum_{K^2} \frac{ \langle |f_{{n+1},K}| w, h_{I_{n+1}^1}^0\otimes h_{K^2}^0\rangle^2}{\langle \gamma_{n+1}\rangle_K^2 }\frac{1_{K^2}}{|K^2|}\Big)^{\frac 12},
\]and $\gamma_{n+1}=v_{n+1}^{-2}$. For brevity, in below we just write $\sum_{I_{n+1}^1}$ instead of $\sum_{(I_{n+1}^1)^{(k_{n+1})}=K^1}$.

So it remains to prove some variant of Proposition \ref{prop:5.8}, which is straightforward. In fact, for the above model case, since $\gamma_{n+1}\in A_{2(n+1)}$, we have
\[
(\gamma_{n+1}^{-\frac 12}, \gamma_{n+1}^{\frac {1}{2n+1}}, \cdots, \gamma_{n+1}^{\frac {1}{2n+1}} )\in A_{(2,\infty, \cdots, \infty)}.
\]
Thus,
\begin{align*}
F_{{n+1},K^1}\le 1_{K^1}\sum_{I_{n+1}^1}\frac{|I_{n+1}^1|^{\frac 12}}{|K^1|}\Big(\sum_{K^2} \langle |f_{{n+1},K}|w, h_{I_{n+1}^1}^0\otimes h_{K^2}^0\rangle^2 \langle \gamma_{n+1}^{-\frac {1}{2n+1}}\rangle_K^{2\cdot (2n+1)} \frac{1_{K^2}}{|K^2|}\Big)^{\frac 12}.
\end{align*}
If $k_{n+1}=0$, we simply have
\[
F_{{n+1},K^1}\le \Big(\sum_{K^2} \big[ M_{\calD}( |f_{{n+1},K}|w, \gamma_{n+1}^{-\frac {1}{2n+1}},\cdots, \gamma_{n+1}^{-\frac{1}{2n+1}})\big]^2\Big)^{\frac 12}.
\]
Then it is just a matter of vector-valued estimates for the multilinear maximal function and we are done. If $k_1>0$, then let $s>1$ be such that $d_1/{s'}$ is sufficiently small, we have
\begin{align*}
&F_{n+1,K^1}\\
&\le 2^{\frac{k_{n+1}d_1 }{s'}} 1_{K^1}\Big( \sum_{I_{n+1}^1}\frac{|I_{n+1}^1|^{\frac s2}}{|K^1|^s}\Big(\sum_{K^2} \frac{ \langle |f_{n+1,K}|w, h_{I_{n+1}^1}^0\otimes h_{K^2}^0\rangle^2}{\langle \gamma_{n+1}\rangle_K^2 }\frac{1_{K^2}}{|K^2|}\Big)^{\frac s2}\Big)^{\frac 1s}\\
&\le 2^{\frac{k_{n+1}d_1 }{s'}} 1_{K^1}\\
&\quad\otimes\Big( \sum_{I_{n+1}^1}\frac{|I_{n+1}^1|^{\frac s2}}{|K^1|^s}\Big(\sum_{K^2} \langle |f_{n+1,K}|w, h_{I_{n+1}^1}^0\otimes h_{K^2}^0\rangle^2 \bla \langle \gamma_{n+1}\rangle_{K^1,1}^{-\frac{1}{2n+1}}\bra_{K^2}^{2\cdot (2n+1)} \frac{1_{K^2}}{|K^2|}\Big)^{\frac s2}\Big)^{\frac 1s}\\
&\le 2^{\frac{k_{n+1}d_1 }{s'}} 1_{K^1}\\
&\quad\otimes\Big( \sum_{I_{n+1}^1}\frac{|I_{n+1}^1|^{\frac s2}}{|K^1|^s}\Big(\sum_{K^2} \big[ M_{\calD^2} (\langle |f_{n+1,K}|w, h_{I_{n+1}^1}^0\rangle, \langle \gamma_{n+1}\rangle_{K^1,1}^{-\frac{1}{2n+1}},\cdots, \langle \gamma_{n+1}\rangle_{K^1,1}^{-\frac{1}{2n+1}})\big]^2\Big)^{\frac s2}\Big)^{\frac 1s}.
\end{align*}
Then the fact that
\[
(\langle\gamma_{n+1}\rangle_{K^1,1}^{-\frac 12}, \langle\gamma_{n+1}\rangle_{K^1,1}^{\frac{1}{2n+1}}, \cdots, \langle\gamma_{n+1}\rangle_{K^1,1}^{\frac{ 1}{2n+1}} )\in A_{(2,\infty, \cdots, \infty)}(\R^{d_2})
\]
with characteristic independent of $K^1$ gives us that
\begin{align*}
&\Big\| \Big(\sum_{K^1}|F_{n+1,K^1}|^2\Big)^{\frac 12} \gamma_{n+1}^{\frac 12}\Big\|_{L^2}^2\\
&\le 2^{\frac{2k_{n+1}d_1 }{s'}}\sum_{K^1}\int_{\R^{d_2}}\langle\gamma_{n+1}\rangle_{K^1,1}\\
&\times\Big( \sum_{I_{n+1}^1}\frac{|I_{n+1}^1|^{\frac s2}}{|K^1|^{\frac s2}}\Big(\sum_{K^2} \big[ M_{\calD^2} (\langle |f_{n+1,K}|w, h_{I_{n+1}^1}^0\rangle, \langle \gamma_{n+1}\rangle_{K^1,1}^{-\frac{1}{2n+1}},\cdots, \langle \gamma_{n+1}\rangle_{K^1,1}^{-\frac{1}{2n+1}})\big]^2\Big)^{\frac s2}\Big)^{\frac 2s} \\
&\lesssim 2^{\frac{2k_{n+1}d_1 }{s'}}\sum_{K^1}\int_{\R^{d_2}}\Big( \sum_{I_{n+1}^1}\frac{|I_{n+1}^1|^{\frac s2}}{|K^1|^{\frac s2}}\Big(\sum_{K^2} \langle |f_{n+1,K}|w, h_{I_{n+1}^1}^0\rangle^2\Big)^{\frac s2}\Big)^{\frac 2s}\langle \gamma_{n+1}\rangle_{K^1,1}^{-1}\\
&\le 2^{\frac{2k_{n+1}d_1 }{s'}}\sum_{K^1}\int_{\R^{d_2}}\Big( \sum_{I_{n+1}^1}\frac{|I_{n+1}^1|^{\frac 12}}{|K^1|^{\frac 12}}\Big(\sum_{K^2} \langle |f_{n+1,K}|w, h_{I_{n+1}^1}^0\rangle^2\Big)^{\frac 12}\Big)^{ 2}\langle \gamma_{n+1}\rangle_{K^1,1}^{-1}.
\end{align*}
By Minkowski's inequality,
\[
\Big(\sum_{K^2} \langle |f_{n+1,K}|w, h_{I_{n+1}^1}^0\rangle^2\Big)^{\frac 12}\le \Big\langle \big(\sum_{K^2}|f_{n+1,K}w|^2\big)^{\frac 12}, h_{I_{n+1}^1}^0\Big\rangle.
\]
We are left with estimating
\begin{align*}
&\sum_{K^1}\int_{\R^{d_2}}\Big( \sum_{(I_{n+1}^1)^{(k_{n+1})}=K^1}\frac{|I_{n+1}^1|^{\frac 12}}{|K^1|^{\frac 12}}\Big\langle \big(\sum_{K^2}|f_{n+1,K}w|^2\big)^{\frac 12}, h_{I_{n+1}^1}^0\Big\rangle\Big)^{ 2}\langle \gamma_{n+1}\rangle_{K^1,1}^{-1}\\
&\hspace{4cm}= \int_{\R^d}\sum_{K^1}\frac{\Big \langle \big(\sum_{K^2}|f_{n+1,K}w|^2\big)^{\frac 12}\Big\rangle_{K^1}^2 }{ \langle \gamma_{n+1}\rangle_{K^1,1}^2}1_{K^1} \gamma_{n+1}\\
&\hspace{4cm}\le \int_{\R^d}\sum_{K^1} \Big \langle \big(\sum_{K^2}|f_{n+1,K}w|^2\big)^{\frac 12}\Big\rangle_{K^1}^2 \langle \gamma_{n+1}^{-\frac{ 1}{2n+1}}\rangle_{K^1,1}^{2\cdot (2n+1)}1_{K^1} \gamma_{n+1}\\
&\hspace{4cm}\lesssim \int_{\R^d}  \sum_{K^1} \big(\sum_{K^2}|f_{n+1,K}w|^2\big)^{ \frac 12\cdot 2}\gamma_{n+1}^{-1}\\
&\hspace{4cm}= \| \widetilde f_{n+1} v_{n+1}\|_{L^2}^2.
\end{align*}
This completes the proof.
The case $p\le 1$ follows from extrapolation \cite{LMMOV}.
\end{proof}

\section{The upper bound}

In this section, we prove the following theorem.
\begin{thm}\label{thm:mainModel}
 Let $\vec p = (p_1, \ldots, p_n)$ so that $1 < p_i \le \infty$, define $1/p = \sum_{i=1}^n 1/p_i > 0.$
Let $(w_1, \ldots, w_n),(\lambda_1, w_2, \ldots, w_n) \in A_{\vec p}$ and let the associated Bloom weight $\nu = w_1 \lambda_1^{-1} \in A_\infty.$  Assume that
$b \in \bmo(\nu).$

For a multilinear bi-parameter dyadic model operator $U,$ defined in the section \ref{sec:model}, we have
$$
\|[b,U]_1(f_1, \ldots, f_n) \nu^{-1}w \|_{L^p} \lesssim_k \| b\|_{\bmo(\nu)} \prod_{i=1}^n \|f_i w_i\|_{L^{p_i}}.
$$
Here the constant depends on the complexity $k = (k_1,\ldots,k_n) = ((k_1^1,k_1^2),\ldots,(k_{n+1}^1,k_{n+1}^2))$ whenever $U$ is a shift or a partial paraproduct. Dependence of the complexity is
\begin{equation}\label{eq:complx}
\begin{cases}
C_\beta 2^{\max_i k_i \beta} \quad \text{ for  every } \beta \in (0,1],\qquad \text{if } U\text{ is a partial paraproduct}\\
(1 + \max\{ k_1^1,k_1^2,k_{n+1}^1 ,k_{n+1}^2\})^{\frac12}, \qquad \text{if } U \text{ is a shift.}
\end{cases}
\end{equation}
\end{thm}
 We divide the analysis of each model operator into different subsections.
 
 The boundedness of these model operator commutators yields the boundedness of the commutators of Calder\'on-Zygmund operators via Proposition \ref{prop:Representation}. Use of Proposition \ref{prop:Representation} and complexity dependences \eqref{eq:complx} restricts the kernel regularity of  $(\omega_1,\omega_2)$-CZOs in Theorem \ref{thm:main}. For the paraproduct free CZOs, we can use milder kernel regularity, where we have that $\omega_i \in \text{Dini}_{3/2}, i = 1,2.$  By paraproduct free, we mean that the paraproducts
in the dyadic representation of $T$ vanish, which could also be stated in terms of
(both partial and full) ``$T1 = 0$'' type conditions. In the
paraproduct free case, the reader can think of convolution form SIOs. 
Otherwise, we must use the standard H\" older type kernel regularity $\omega_i(t) = t^{\alpha_i}, \alpha_i\in (0,1]$. 

In the proof, we consider the boundedness $\prod_{i = 1}^n L^{p_i}(w_i^{p_i}) \to L^{p}(\nu^{-p} w^p)$ for $p > 1$ since Theorem \ref{thm:extrapo} will extend the result to the quasi-Banach range. Recall the notation of dual weights: $\sigma_i = w_i^{-p_i'}, \sigma_{n+1} = (\nu^{-1}w )^{p},$ and $\eta_1 = \lambda_1^{-p_1'}.$ Here we chose to consider the commutators acting on the first function slot as the other ones are symmetrical.

\subsection*{The shift case}
We consider the following commutator
$$
[b, S_{k}]_1(f_1,\ldots,f_n)= bS_{k}(f_1,\ldots,f_n) - S_{k}(bf_1,\ldots, f_n),
$$
where $S_{k} := S_{(k_1,\ldots, k_{n+1})}^{1,2}$ is a standard multilinear bi-parameter shift.

The idea is to expand the commutator so that a product $bf$ paired with Haar functions is expanded in the bi-parameter fashion only if both of the Haar functions are cancellative. In a mixed situation, we expand only in $\R^{d_1}$ or $\R^{d_2}$, and in the remaining fully non-cancellative situation we do not expand at all. This strategy has been important in the recent multi-parameter results -- see e.g. \cites{EA, AMV, LMV:Bloom, LMV:Bloom2}.

We focus on a commutator, where the cancellation appears in a mixed situation on first and last slots, that is, we have a commutator that is expanded as follows
\begin{equation}
\sum_{j_1=1}^3 \Pi^{1}_{j_1}(b,S_k(f_1, \ldots, f_{n})) - \sum_{j_2 =1}^3S_k(\Pi_{j_2}^2(b,f_1),\ldots, f_n).\label{eq:Case1}
\end{equation}
This case essentially gathers all the methods for estimating these commutators. More involved expansions are considered with partial paraproducts.

Both terms are handled separately whenever we have a bounded paraproduct, that is $\Pi_{j_i}^i, j_i \neq 3$ (or bi-parameter $\Pi_{j_1,j_2}$, $(j_1,j_2) \neq (3,3)$). Otherwise, we need to add and subtract certain averages of the function $b$ to obtain enough cancellation. We analyse the second term in \eqref{eq:Case1} as the first term is similar (swap the roles of functions $f_1$ and $f_{n+1}$ together with weights $\eta_1$ and $w^p$).

We begin with the term
\begin{align*}
&S_k( \Pi_{1}^2(b,f_1),\ldots, f_n) \\
&= \sum_{K^1\times K^2 \in \calD^1 \times \calD^2} \sum_{\substack{I_i^j \in \calD^j\\ (I_i^j)^{(k_i^j)} = K^j\\ i= 1,\ldots,n+1, j =1,2}} &a_{K(I_i^j)}\Big\langle\sum_{J^2 \in \calD^2} \ave{b,h_{J^2}}_2 \ave{f_1, h_{J^2}}_2 \otimes h_{J^2}h_{J^2} , h_{I_1^1}^0 \otimes h_{I_1^2} \Big\rangle \\
&&\times \prod_{i=2}^n \ave{f_{i},\wt h_{I_i}} h_{I_{n+1}^1} \otimes h_{I_{n+1}^2}^0.
\end{align*}
By the zero average of Haar functions, we always have $J^2 \subset I_1^2$. Now the important observation is that when $J^2 \subsetneq I_1^2$ we must have $h_{J^2}h_{J^2}=\frac{1_{J^2}}{|J^2|}$ and we can replace $\ave{\frac{1_{J^2}}{|J^2|}, h_{I_1^2}}$ with $\ave{ \frac{\eta_1 1_{J^2}}{\eta_1(J^2)}, h_{I_1^2}}$. Thus, we can change the order of the operators, and we can split the dual form of the term as follows
\begin{align*}
&\ave{S_k( \Pi_{1}^2(b,f_1),\ldots, f_n),f_{n+1}} \\
&= \sum_{K \in \calD} \sum_{\substack{I_i^j \in \calD^j\\ (I_i^j)^{(k_i^j)} = K^j\\ i= 1,\ldots, n+1 , j =1,2}} a_{K(I_i^j)}\Big\langle\sum_{J^2\subsetneq I_1^2\in \calD^2} \ave{b,h_{J^2}}_2 \ave{f_1, h_{J^2}}_2  \frac{ \eta_1 1_{J^2}}{\eta_1(J^2)}, h_{I_1^1}^0 \otimes h_{I_1^2} \Big\rangle \\
&\qquad\times\prod_{i=2}^n \ave{f_{i},\wt h_{I_i}}\ave{f_{n+1}, h_{I_{n+1}^1} \otimes h_{I_{n+1}^2}^0 }+ \sum_{K\in \calD} \sum_{\substack{I_i^j \in \calD^j\\ (I_i^j)^{(k_i^j)} = K^j\\ i= 1,\ldots, n+1 , j =1,2}} a_{K(I_i^j)} \Big\langle \ave{b,h_{I_1^2}}_2 \ave{f_1, h_{I_1^2}}_2, h_{I_1^1}^0\Big\rangle \\
&\qquad\times\langle h_{I_1^2}h_{I_1^2}, h_{I_1^2}\rangle\prod_{i=2}^n \ave{f_{i},\wt h_{I_i}}\ave{f_{n+1}, h_{I_{n+1}^1} \otimes h_{I_{n+1}^2}^0 }
\\
&=\int_{\R^{d_1}} \sum_{J^2\in \calD^2} \ave{b,h_{J^2}}_2 \ave{f_1, h_{J^2}}_2 \ave{S^{1*}_k(f_2,\ldots, f_{n+1})}_{J^2,2}^{\eta_1}\\
&\hspace{4cm}-\int_{\R^{d_1}} \sum_{J^2\in \calD^2} \ave{b,h_{J^2}}_2 \ave{f_1, h_{J^2}}_2\ave{S_{k,J^2}^{1*}(f_2,\ldots, f_{n+1})}_{J^2,2}^{\eta_1}+E, \\
\end{align*}
 where $S_{k,J^2}^{1*}$ differs from the usual adjoint so that we have $I_1^2\subset J^2.$ We do not explicitly handle the term $E$ as it is similar to the case $j_2=2$ (note that we have more cancellation than we need). Since the truncated operator $S_{k,J^2}^{1*}$ can be dominated by the $A_{\infty}$ weighted square functions lower bound, we can drop the dependence on cube $J^2.$ Then, the estimations of the first two terms are very similar, hence one might think of $S_{k}^{1*}$ as such or as $S_\calD^2 S_k^{1*}$ below.  The boundedness follows simply by using Proposition \ref{prop:bloomParaEta1} and the boundedness of multilinear shifts. Namely,
\begin{align*}
\int_{\R^{d_1}} \sum_{J^2\in \calD^2} \ave{b,h_{J^2}}_2 \ave{f_1, h_{J^2}}_2 &\ave{S_k^{1*}(f_2,\ldots, f_{n+1})}_{J^2,2}^{\eta_1} \\
&\lesssim \|b\|_{\bmo(\nu)}\|f_1w_1\|_{L^{p_1}} \|  S_k^{1*}(f_2,\ldots, f_{n+1})\|_{L^{p_1'}(\eta_1)}\\
&\lesssim \|b\|_{\bmo(\nu)}\prod_{i=1}^n\|f_i w_i\|_{L^{p_i}} \|f_{n+1} \nu w^{-1}\|_{L^{p'}}
\end{align*}
since $\eta_1^{1/p_1'} = \lambda_1^{-1} = \nu w^{-1} \prod_{i =2}^n w_i $ and $(w_2,\ldots, w_n, \nu w^{-1} )\in A_{(p_2,\ldots, p_n, p')}$ by Lemma \ref{lem:lem7}.

The term, where $j_2  = 2,$ is significantly more straightforward to estimate. We consider the dual form and estimate 
\begin{align*}
&|\ave{S(\Pi_{2}^2(b,f_1),\ldots, f_n),f_{n+1}}| \\
&= \Big|\sum_{K \in \calD} \sum_{\substack{I_i^j \in \calD^j\\ (I_i^j)^{(k_i^j)} = K^j\\ i= 1,\ldots,n+1, j =1,2}} a_{K(I_i^j)}\Big\langle \ave{b,h_{I_1^2}}_2 \ave{f_1}_{I_1^2,2} , h_{I_1^1}^0 \Big\rangle 
\prod_{i=2}^n \ave{f_{i},\wt h_{I_i}}\ave{f_{n+1}, h_{I_{n+1}^1} \otimes h_{I_{n+1}^2}^0 }\Big|\\
&\leq \int_{\R^{d_1}}\sum_{K^2} \sum_{(I_1^2)^{(k_1^2)} = K^2} |I_1^2|^{1/2}|\ave{b,h_{I_1^2}}_2| \ave{| f_1|}_{I_1^2,2} A_{K^2, (k_{n+1}^1,k_{i_1^1}^{1},k_{i_1^2}^2)}(f_2,\ldots, f_{n+1})\\
&= \int_{\R^{d_1}}\sum_{K^2} \sum_{(I_1^2)^{(k_1^2)} = K^2} |I_1^2|^{1/2}|\ave{b,h_{I_1^2}}_2\ave{\sigma_1}_{I_1^2,2}| \frac{\ave{| f_1|}_{I_1^2,2}}{\ave{\sigma_1}_{I_1^2,2}} A_{K^2, (k_{n+1}^1,k_{i_1^1}^{1},k_{i_1^2}^2)}(f_2,\ldots, f_{n+1})
\end{align*}
where 
$A_{K^2, (k_{n+1}^1,k_{i_1^1}^{1},k_{i_1^2}^2)}, i_1^m \in \{2,\ldots, n\}$ is from family of operators such that the square sum
$$\Big(\sum_{K^2} A_{K^2, (k_{n+1}^1,k_{i_1^1}^{1},k_{i_1^2}^2)}^2(f_2,\ldots, f_{n+1}) 1_{K^2} \Big)^{\frac 12}$$ is an $A_{2,k}$ type square function.
We use Lemma \ref{lem:MWuniform} with a fixed variable on the first parameter and get
\begin{align*}
&|\ave{S(\Pi_{2}^2(b,f_1),\ldots, f_n),f_{n+1}}| \\
&\lesssim \|b\|_{\bmo(\nu)} \int \Big(\sum_{K^2} \sum_{(I_1^2)^{(k_1^2)} = K^2}\frac{\ave{| f_1|}_{I_1^2,2}^2}{\ave{\sigma_1}_{I_1^2,2}^2} A^2_{K^2,(k_{n+1}^1,k_{i_1^1}^{1},k_{i_1^2}^2)}(f_2,\ldots, f_{n+1}) 1_{I_1^2} \Big)^{1/2} \sigma_1 \nu\\
&\leq \|b\|_{\bmo(\nu)} \int M_{\calD^2}^{\sigma_1}( f_1\sigma_1^{-1} ) A_{2,(k_{n+1}^1,k_{i_1^1}^{1},k_{i_1^2}^2)}(f_2,\ldots, f_{n+1})  \sigma_1 \nu \\
&\leq \|b\|_{\bmo(\nu)}\| M_{\calD^2}^{\sigma_1}( f_1\sigma_1^{-1} )\|_{L^{p_1}(\sigma_1)} \| A_{2,(k_{n+1}^1,k_{i_1^1}^{1},k_{i_1^2}^2)}(f_2,\ldots, f_{n+1}) \lambda_1^{-1} \|_{L^{p_1'}} \\
&\lesssim \|b\|_{\bmo(\nu)}\|f_1 w_1 \|_{L^{p_1}} \Big\| A_{2,(k_{n+1}^1,k_{i_1^1}^{1},k_{i_1^2}^2)}(f_2,\ldots, f_{n+1}) \nu w^{-1} \prod_{i=2}^n w_i \Big\|_{L^{p_1'}} \\
&\lesssim \|b\|_{\bmo(\nu)} \prod_{i=1}^n\|f_i w_i\|_{L^{p_i}} \|f_{n+1} \nu w^{-1}\|_{L^{p'}}.
\end{align*}
In the above estimates, it is enough to note that the maximal function is bounded since, by Fubini's theorem, we can work with a fixed variable on the first parameter and use the classical one-parameter result.

Lastly, we are left with the paraproducts of the illegal form
\begin{align*}
&\ave{\Pi^{1}_{3}(b,S_k(f_1, \ldots, f_{n})),f_{n+1}} - \ave{S_k(\Pi_{3}^2(b,f_1),\ldots, f_n),f_{n+1}} \\
&= \sum_{K} \sum_{\substack{I_i^j \in \calD^j\\ (I_i^j)^{(k_i^j)} = K^j\\ i= 1,\dots,n+1, j =1,2}} a_{K(I_i^j)}   \ave{f_1, h_{I_1^1}^0 \otimes h_{I_1^2}} \prod_{i=2}^n \ave{f_{i},\wt h_{I_i}} \Ave{\ave{b}_{I_{n+1}^1,1}f_{n+1}, h_{I_{n+1}^1} \otimes h_{I_{n+1}^2}^0 } \\
&\qquad -\sum_{K} \sum_{\substack{I_i^j \in \calD^j\\ (I_i^j)^{(k_i^j)} = K^j\\ i= 1,\dots,n+1, j =1,2}} a_{K(I_i^j)}   \Ave{\ave{b}_{I_1^2,2}f_1, h_{I_1^1}^0 \otimes h_{I_1^2}}  \prod_{i=2}^n \ave{f_{i},\wt h_{I_i}} \ave{f_{n+1}, h_{I_{n+1}^1} \otimes h_{I_{n+1}^2}^0 } .
\end{align*}
Here we introduce the martingale blocks to the function $b.$ We write
\begin{align*}
\ave{b}_{I_{n+1}^1,1} &= \ave{b}_{I_{n+1}^1,1} - \ave{b}_{I_{n+1}^1 \times I_{n+1}^2}+\ave{b}_{I_{n+1}^1 \times I_{n+1}^2}- \ave{b}_{I_{n+1}^1 \times K^2}\\
&\hspace{6cm}+ \ave{b}_{I_{n+1}^1 \times K^2} - \ave{b}_{K^1 \times K^2}+\ave{b}_{K^1 \times K^2}
\end{align*}and likewise for $\ave{b}_{I_{1}^2,2}$. The extra $\langle b \rangle_{K^1\times K^2}$ simply cancels with the one from $\ave{b}_{I_{1}^2,2}$.
 Hence, in the commutator we can expand as follows
\begin{align}
(\ave{b}_{I_{n+1}^1,1} - \ave{b}_{I_{n+1}^1 \times I_{n+1}^2})1_{I_{n+1}^2} &= \sum_{J^2 \subset I_{n+1}^2} \Ave{b, \frac{1_{I_{n+1}^1}}{|I_{n+1}^1|} \otimes h_{J^2} } h_{J^2}, \label{eq:shiftM1}\\
\ave{b}_{I_{n+1}^1 \times K^2} - \ave{b}_{K^1 \times K^2} &= \sum_{I_{n+1}^1 \subsetneq J^1 \subset K^1}\Ave{b,h_{J^1} \otimes \frac{1_{K^2}}{|K^2|}}\ave{h_{J^1}}_{I_{n+1}^1}, \label{eq:shiftM2} \\
\ave{b}_{I_{n+1}^1 \times I_{n+1}^2}-\ave{b}_{I_{n+1}^1 \times K^2}  &=  \sum_{I_{n+1}^2 \subsetneq J^2 \subset K^2} \Ave{b, \frac{1_{I_{n+1}^1}}{|I_{n+1}^1|} \otimes h_{J^2}}\ave{h_{J^2}}_{I_{n+1}^2}. \label{eq:shiftM3}
\end{align}

Observe that we have omitted the terms raised from $\ave{b}_{I_{1}^2,2}$ because they are similar. On the other hand, we shall only work with \eqref{eq:shiftM1} and \eqref{eq:shiftM2} because \eqref{eq:shiftM3} is analogous. 

We begin with the dual form of \eqref{eq:shiftM1}
\begin{align*}
& \Big|\sum_{K} \sum_{\substack{I_i^j \in \calD^j\\ (I_i^j)^{(k_i^j)} = K^j\\ i= 1,\dots,n+1, j =1,2}} a_{K(I_i^j)} \sum_{J^2\subset I_{n+1}^2} \Ave{b,\frac{1_{I_{n+1}^1}}{|I_{n+1}^1|}\otimes h_{J^2}} |I_{n+1}^2|^{-\frac 12} \ave{f_1, h_{I_1^1}^0 \otimes h_{I_1^2}} \\
&\hspace{8cm}\times \prod_{i=2}^n \ave{f_{i},\wt h_{I_i}} \ave{f_{n+1}, h_{I_{n+1}^1} \otimes h_{J^2}} \Big|. 
\end{align*}By similar arguments as that in the proof of Lemma \ref{lem:MWestimate}, we have
\begin{align*}
&\sum_{I_{n+1}^{(k_{n+1})}=K}|I_{n+1}^1|^{\frac 12}|I_{n+1}^2|^{\frac 12}\sum_{J^2\subset I_{n+1}^2} \Ave{b,\frac{1_{I_{n+1}^1}}{|I_{n+1}^1|}\otimes h_{J^2}} |I_{n+1}^2|^{-\frac 12}  \ave{f_{n+1}, h_{I_{n+1}^1} \otimes h_{J^2}} \\
&\qquad\lesssim  \|b\|_{\bmo(\nu)}\sum_{I_{n+1}^{(k_{n+1})}=K}\int_{\R^{d_1}} h^0_{I_{n+1}^1} \int_{\R^{d_2}} \Big(\sum_{J^2\subset I_{n+1}^2}   \langle f_{n+1}, h_{I_{n+1}^1}\otimes h_{J^2}\rangle^2  \frac{1_{J^2}}{|J^2|}\Big)^{\frac 12} \nu\\
&\qquad\lesssim \|b\|_{\bmo(\nu)}\sum_{(I_{n+1}^1)^{(k_{n+1}^1)}=K^1}\int_{\R^d} h^0_{I_{n+1}^1}\otimes 1_{K^2}\Big(\sum_{J^2}   \frac{\langle f_{n+1}, h_{I_{n+1}^1}\otimes h_{J^2}\rangle^2}{\langle \sigma_{n+1}\rangle_{I_{n+1}^1\times J^2}^2}  \frac{1_{J^2}}{|J^2|}\Big)^{\frac 12}\sigma_{n+1}\nu.
\end{align*}
Then by standard calculus, we can reduce the problem to 
\begin{align*}
\int_{\R^d} \sum_{K^1} A_{K^1,k_1^2, k_{i_1^1}^1, k_{i_1^2}^2}(f_1,\cdots, f_n)1_{K^1}\hspace{-0.5cm}\sum_{(I_{n+1}^1)^{(k_{n+1}^1)}=K^1}h^0_{I_{n+1}^1}\Big(\sum_{J^2}   \frac{\langle f_{n+1}, h_{I_{n+1}^1}\otimes h_{J^2}\rangle^2}{\langle \sigma_{n+1}\rangle_{I_{n+1}^1\times J^2}^2}  \frac{1_{J^2}}{|J^2|}\Big)^{\frac 12}\sigma_{n+1}\nu,
\end{align*}
where $A_{K^1,k_1^2, k_{i_1^1}^1, k_{i_1^2}^2}(f_1,\cdots, f_n)$ is defined such that 
\[
\Big(\sum_{K^1} \big[ A_{K^1, k_1^2,k_{i_1^1}^1, k_{i_1^2}^2}(f_1,\cdots, f_n)\big]^2 1_{K^1}\Big)^{\frac 12}
\]is an $A_{2,k}$ type square function. 
Notice that $\sigma_{n+1}\nu = (\nu^{-1}w)^{p}\nu = (w^p)^\frac{1}{p}((\nu^{-1}w)^{p})^\frac{1}{p'} \in A_\infty.$
The rest follows from estimates such as H\"older's inequality, Theorem \ref{thm:multilinSF}, and Proposition \ref{prop:5.8}.

Finally, we consider the dual form of \eqref{eq:shiftM2}
\begin{align*}
&\Big|\sum_{K} \sum_{\substack{I_i^j \in \calD^j\\ (I_i^j)^{(k_i^j)} = K^j\\ i= 1,\ldots,n+1, j =1,2}} a_{K(I_i^j)} \sum_{I_{n+1}^1 \subsetneq J^1 \subset K^1}\Ave{b,h_{J^1} \otimes \frac{1_{K^2}}{|K^2|} }\ave{h_{J^1}}_{I_{n+1}^1}\prod_{i=1}^{n+1} \ave{f_{i},\wt h_{I_i}} \Big| \\
&\leq \int_{\R^{d_2}} \sum_{K^1} \sum_{\substack{J^1 \subset K^1\\ \ell(J^1) > 2^{-k_{n+1}^1}\ell(K^1)}} |J^1|^{-\frac{1}{2}} |\ave{b,h_{J^1} }_1 | \sum_{K^2}\sum_{\substack{I_i^j \in \calD^j\\ (I_i^j)^{(k_i^j)} = K^j\\
I_{n+1}^1 \subset J^1\\ i= 1,\ldots,n+1, j =1,2}} |a_{K(I_i^j)} |\prod_{i=1}^{n+1} |\ave{f_{i},\wt h_{I_i}} | \frac{1_{K^2}}{|K^2|} \\
&\leq \int_{\R^{d_2}} \sum_{K^1} \sum_{j_{n+1}^1=0}^{k_{n+1}^1-1}\sum_{(J^1)^{(j_{n+1}^1)}=K^1}  |J^1|^{\frac{1}{2}} |\ave{b,h_{J^1} }_1 |\\
&\hspace{2cm}\times \sum_{K^2} A_{K, (k_1^2, k_{i_1^1}^1,k_{i_1^2}^2)}(f_1, \ldots, f_{n}) \ave{|\Delta_{J^1, k_{n+1}^1-j_{n+1}^1}^1 f_{n+1}|}_{J^1\times K^2} 1_{K^2},
\end{align*}
where $ A_{K, (k_1^2, k_{i_1^1}^1,k_{i_1^2}^2)}$ is defined such that 
$$
\Big(\sum_{K^1}  \Big(\sum_{K^2}A_{K, (k_1^2, k_{i_1^1}^1,k_{i_1^2}^2)}(f_1, \ldots, f_{n}) 1_{K}\Big)^2\Big)^{\frac 12}
$$
is an $A_{2,k}$ type square function.

This resembles the term that we faced earlier with paraproduct $\Pi_2.$ The only meaningful difference is the extra summation. The estimations are similar when we divide and multiply with $\ave{\sigma_{n+1}}_{J^1\times K^2}$. To be more precise, that is, we write 
\begin{align*}
&\int_{\R^{d_2}}\sum_{j_{n+1}^1=0}^{k_{n+1}^1-1}\sum_{(J^1)^{(j_{n+1}^1)}=K^1}  |J^1|^{\frac{1}{2}} |\ave{b,h_{J^1} }_1 |\ave{|\Delta_{J^1,k_{n+1}^1-j_{n+1}^1}^1 f_{n+1}|}_{J^1\times K^2}1_{K^2}\\
&\hspace{1.5cm}=\int_{\R^{d_2}}\sum_{\substack{(J^1)^{(j_{n+1}^1)}=K^1\\ 0\le j_{n+1}^1\le k_{n+1}^1-1}}|J^1|^{\frac 12} |\ave{b,h_{J^1} }_1 |\ave{\sigma_{n+1}}_{J^1\times K^2}\frac{\ave{|\Delta_{J^1,k_{n+1}^1-j_{n+1}^1}^1 f_{n+1}|}_{J^1\times K^2}}{\ave{\sigma_{n+1}}_{J^1\times K^2}}1_{K^2}\\
&\hspace{1.5cm}\lesssim \|b\|_{\bmo(\nu)}\int_{\R^{d}}\Big(\sum_{\substack{(J^1)^{(j_{n+1}^1)}=K^1\\ 0\le j_{n+1}^1\le k_{n+1}^1-1}}\frac{\ave{|\Delta_{J^1,k_{n+1}^1-j_{n+1}^1}^1 f_{n+1}|}_{J^1\times K^2}^2}{\ave{\sigma_{n+1}}_{J^1\times K^2}^2} 1_{J^1} \Big)^{\frac 12}\nu\sigma_{n+1}1_{K^2},
\end{align*}
where we have used Lemma \ref{lem:MWestimate}. 
The rest of the argument is rather standard and thus the object is bounded by
\begin{align*}
(1+ k_{n+1}^1)^{\frac12}\|b\|_{\bmo(\nu)}\prod_{i=1}^n\|f_i w_i\|_{L^{p_i}} \|f_{n+1} \nu w^{-1}\|_{L^{p'}},
\end{align*}
where dependence $(1+ k_{n+1}^1)^{\frac12}$ emerges from the summation of $
0\le j_{n+1}^1\le k_{n+1}^1-1.
$
This completes the analysis of the commutator of this form.

Although other forms of shifts lead to different expansions, the methods shown above are sufficient to handle those as well.  Since we are dealing with multilinear shifts, we now encounter terms in the shift case that are non-cancellative. In comparison, this does not happen in the linear case in \cite{LMV:Bloom}, where we always expand in the bi-parameter fashion.
For example, if we look at the term $bS(f_1,\ldots, f_{n})- S(\Pi_{3,3}^{1,2}(b,f_1),\ldots, f_{n}),$ we have
$$
(b - \ave{b}_{I_1^1 \times I_1^2}) 1_{I_{n+1}^1 \times I_{n+1}^2}.
$$
We write
\begin{align*}
(b - \ave{b}_{I_1^1 \times I_1^2}) 1_{I_{n+1}^1 \times I_{n+1}^2} &=( (b - \ave{b}_{I_{n+1}^1 \times I_{n+1}^2})  + (\ave{b}_{I_{n+1}^1 \times I_{n+1}^2} - \ave{b}_{I_1^1 \times I_1^2}) ) 1_{I_{n+1}^1 \times I_{n+1}^2} \\
&= (b -\ave{b}_{I_{n+1}^1,1} - \ave{b}_{I_{n+1}^2,2} + \ave{b}_{I_{n+1}^1 \times I_{n+1}^2}) 1_{I_{n+1}^1 \times I_{n+1}^2} \\
&\qquad+(\ave{b}_{I_{n+1}^1,1} - \ave{b}_{I_{n+1}^1 \times I_{n+1}^2}) 1_{I_{n+1}^1 \times I_{n+1}^2}  \\
&\qquad+ (\ave{b}_{I_{n+1}^2,2} - \ave{b}_{I_{n+1}^1 \times I_{n+1}^2}) 1_{I_{n+1}^1 \times I_{n+1}^2} \\
&\qquad+ (\ave{b}_{I_{n+1}^1 \times I_{n+1}^2} - \ave{b}_{I_1^1 \times I_1^2}) 1_{I_{n+1}^1 \times I_{n+1}^2} .
\end{align*}
The above terms are expanded to the martingale blocks and differences in a standard way like terms \eqref{eq:shiftM1} and \eqref{eq:shiftM2}. Note that the first term on the right-hand side produces a bi-parameter martingale difference inside of the rectangle $I_{n+1}^1 \times I_{n+2}^2.$ We will analyse similar terms in the following subsection.

\subsection*{Partial paraproducts}
As explained earlier, we will now focus on more involved expansions of the commutator. We show the most representative case out of those. Although we demonstrated the main ideas of the estimates already in the shift case, we need to use more complex estimates due to the more complicated structure of the partial paraproducts.

 We do not repeat the expansion strategy and instead straight away consider separately 
\begin{align*}
&\ave{(S\pi)(\Pi^{1,2}_{j_1,j_2}(b,f_1),\ldots, f_{n}),f_{n+1}} \\
&=\sum_{K^1, K^2} \sum_{(I_i^1)^{(k_i)} = K^1} a_{K(I_i^1)}\ave{\Pi^{1,2}_{j_1,j_2}(b,f_1), h_{I_1^1} \otimes h_{K^2}} \prod_{i=2}^{n+1} \Ave{f_i,\wt h_{I_i^1} \otimes \frac{1_{K^2}}{|K^2|}}
\end{align*}
for $(j_1,j_2)\neq (3,3).$ 
We collect most of the mixed index $(j_1 \neq j_2)$ cases, as the methods can be attained from these.

Let us begin with the term, where $j_1 = 1, j_2 = 2,$ that equals 
\begin{align*}
\sum_{K^1, K^2} &\sum_{(I_i^1)^{(k_i)} = K^1} a_{K(I_i^1)}\ave{\Pi^{1,2}_{1,2}(b,f_1), h_{I_1^1} \otimes h_{K^2}} \prod_{i=2}^{n+1} \Ave{f_i,\wt h_{I_i^1} \otimes \frac{1_{K^2}}{|K^2|}} \\
&= \sum_{J^1, K^2} \ave{b,h_{J^1} \otimes h_{K^2}} \Ave{f_1, h_{J^1} \otimes \frac{1_{K^2}}{|K^2|}}\\
&\hspace{3cm}\times\Ave{ \sum_{K^1}\sum_{(I_i^1)^{(k_i)} = K^1} a_{K(I_i^1)} \prod_{i=2}^{n+1} \Ave{f_i,\wt h_{I_i^1} \otimes \frac{1_{K^2}}{|K^2|}} h_{I_1^1}}_{J^1}^{\ave{\eta_1}_{K^2,2}} \\
&- \sum_{J^1, K^2} \ave{b,h_{J^1} \otimes h_{K^2}} \Ave{f_1, h_{J^1} \otimes \frac{1_{K^2}}{|K^2|}}\\
&\hspace{3cm}\times \Ave{ \sum_{K^1}\sum_{\substack{(I_i^1)^{(k_i)} = K^1\\ I_1^1 \subset J^1}} a_{K(I_i^1)} \prod_{i=2}^{n+1} \Ave{f_i,\wt h_{I_i^1} \otimes \frac{1_{K^2}}{|K^2|}} h_{I_1^1}}_{J^1}^{\ave{\eta_1}_{K^2,2}}\\
&+ \sum_{K^1, K^2} \sum_{(I_i^1)^{(k_i)} = K^1} a_{K(I_i^1)}\langle b, h_{I_1^1}\otimes h_{K^2}\rangle |I_1^1|^{-\frac 12}\prod_{i=1}^{n+1} \Ave{f_i,\wt h_{I_i^1} \otimes \frac{1_{K^2}}{|K^2|}}. 
\end{align*}
Similarly to the previously seen techniques, for the second term we use the square function lower bound to get rid of the restriction $I_1^1\subset J^1.$ Thus, via Proposition \ref{prop:bloomParaMixedEta} we can bound the first two terms by
$$
\|b\|_{\bmo(\nu)}\|f_1w_1\|_{L^{p_1}} \| S_\calD(S\pi)_k(f_2,\ldots, f_{n+1})\lambda_1^{-1}\|_{L^{p_1'}}.
$$ 
Clearly, Lemma \ref{lem:SFpartial} is enough to conclude the claim. The estimate for the remaining term is easier. We apply Lemma \ref{lem:MWestimate} and note that we have more cancellation than we need. Hence, we control 
\[
|I_1^1|^{-\frac 12}\Big|\Ave{ f_1, h_{I_1^1}\otimes \frac {1_{K^2}}{|K^2|}}\Big|\langle \sigma_1\rangle_{I_1^1\times K^2}^{-1} 1_{I_1^1\times K^2}\le M_{\calD}^{\sigma_1}(f_1\sigma_1^{-1}) 1_{I_1^1\times K^2}.
\]
Thus, we are left to estimate
\[
\|b\|_{\bmo(\nu)}\| M_{\calD}^{\sigma_1}(f_1\sigma_1^{-1}) S_\calD(S\pi)_k(f_2,\ldots, f_{n+1})\sigma_1 \nu\|_{L^{1}}.
\]
The desired estimate follows by H\"older's inequality and Lemma \ref{lem:SFpartial}. We remark that the remaining term essentially contains the idea to handle $\Pi_{1,1}$.

The term with $\Pi_{2,1}$ is analogous to the previous one. We remark that in this case, the weighted paraproduct operator has the weight $\ave{\eta_1}_{I_1^1}$ as the localization of the operator is at that level on the first parameter. The cases $\Pi_{3,1}$ and $\Pi_{1,3}$ can be handled similarly. For the sake of the completeness, we give a sketch of the case $\Pi_{3,1}.$ As before, we write 
\begin{align*}
\ave{(S\pi)&(\Pi_{3,1}(b,f_1),\ldots, f_{n}),f_{n+1}} \\
 &=\sum_{K^1, J^2}\sum_{(I_1^1)^{(k_1)} = K^1} \Ave{b,\frac{1_{I_1^1}}{|I_1^1|} \otimes h_{J^2}} \ave{f_1, h_{I_1^1}  \otimes h_{J^2}}\\
&\hspace{4cm}\times\Ave{ \sum_{K^2 \supsetneq J^2}\sum_{\substack{(I_i^1)^{(k_i)} = K^1\\i\neq 1}} a_{K(I_i^1)} \prod_{i=2}^{n+1} \Ave{f_i,\wt h_{I_i^1} \otimes \frac{1_{K^2}}{|K^2|}} h_{K^2}}_{J^2} \\
&+  \sum_{K^1, K^2} \sum_{(I_i^1)^{(k_i)} = K^1} a_{K(I_i^1)}\langle b, \frac{1_{I_1^1}}{|I_1^1|}\otimes h_{K^2}\rangle |K^2|^{-\frac 12}\langle f_1, h_{I_1^1}\otimes h_{K^2}\rangle\prod_{i=2}^{n+1} \Ave{f_i,\wt h_{I_i^1} \otimes \frac{1_{K^2}}{|K^2|}}.
\end{align*}
For the second term,  we again use Lemma \ref{lem:MWestimate} and treat 
\[
 |K^2|^{-\frac 12}|\langle f_1, h_{I_1^1}\otimes h_{K^2}\rangle | \langle \sigma_1\rangle_{I_1^1\times K^2}^{-1}1_{I_1^1\times K^2}\le M_{\calD^2}^{\langle \sigma_1\rangle_{I_1^1,1}}(\langle f_1, h_{I_1^1}\rangle_1\langle \sigma_1\rangle_{I_1^1,1}^{-1} ) 1_{I_1^1\times K^2}.
\]Then after applying H\"older's inequality twice we reduce the problem to bounding
\begin{align*}
&\|b\|_{\bmo(\nu)}\Big\|\Big(\sum_{K^1}\sum_{(I_1^1)^{(k_1)}=K^1}\big[M_{\calD^2}^{\langle \sigma_1\rangle_{I_1^1,1}}(\langle f_1, h_{I_1^1}\rangle_1\langle \sigma_1\rangle_{I_1^1,1}^{-1} )\big]^21_{I_1^1}\Big)^{\frac 12}\Big\|_{L^{p_1}(\sigma_1)}\\
&\hspace{4cm}\times\|S_\calD (S\pi)_k(f_2,\ldots, f_{n+1})\lambda_1^{-1}\|_{L^{p_1'}}.
\end{align*} The estimate is done by Proposition \ref{prop:5.8} and Lemma \ref{lem:SFpartial}. For the first term, we split as usual to
\begin{align*}
& \sum_{K^1, J^2}\sum_{(I_1^1)^{(k_1)} = K^1}  \Ave{b,\frac{1_{I_1^1}}{|I_1^1|}  \otimes h_{J^2}} \ave{f_1, h_{I_1^1}  \otimes h_{J^2}}\\
&\hspace{4cm}\times\Ave{\sum_{K^2 }\sum_{\substack{(I_i^1)^{(k_i)} = (I_1^1)^{(k_1)}\\i\neq1}} a_{K(I_i^1)} \prod_{i=2}^{n+1} \Ave{f_i,\wt h_{I_i^1} \otimes \frac{1_{K^2}}{|K^2|}} h_{K^2} }_{J^2}^{\ave{\eta_1}_{I_1^1,1}} \\
&- \sum_{K^1, J^2}\sum_{(I_1^1)^{(k_1)} = K^1} \Ave{b,\frac{1_{I_1^1}}{|I_1^1|}  \otimes h_{J^2}} \ave{f_1, h_{I_1^1}  \otimes h_{J^2}}\\
&\hspace{4cm}\times\Ave{\sum_{K^2 \subset J^2}\sum_{\substack{(I_i^1)^{(k_i)} = (I_1^1)^{(k_1)}\\i\neq1}} a_{K(I_i^1)} \prod_{i=2}^{n+1} \Ave{f_i,\wt h_{I_i^1} \otimes \frac{1_{K^2}}{|K^2|}} h_{K^2} }_{J^2}^{\ave{\eta_1}_{I_1^1,1}}.
\end{align*}
We focus on the first term as the other one is very similar once square function lower bound is applied inside of the average over $J^2.$ Rewrite the first term as 
\begin{align*}
\sum_{J^1, J^2}\Ave{b,\frac{1_{J^1}}{|J^1|}  \otimes h_{J^2}} \ave{f_1, h_{J^1}  \otimes h_{J^2}}\Ave{\langle (S\pi)_k (f_2, \cdots, f_{n+1}), h_{J^1}\rangle_1}_{J^2}^{\langle \eta_1\rangle_{J^1,1}}.
\end{align*}
Then the estimate is done by Proposition \ref{prop:bloomParaMixedEta} and Lemma \ref{lem:SFpartial}.

We continue with the term, where $j_1 = 2, j_2 = 3$, that is,
\begin{align*}
&\sum_{K^1, K^2} \sum_{(I_i^1)^{(k_i)} = K^1} a_{K(I_i^1)}\ave{\Pi^{1,2}_{2,3}(b,f_1), h_{I_1^1} \otimes h_{K^2}} \prod_{i=2}^{n+1} \Ave{f_i,\wt h_{I_i^1} \otimes \frac{1_{K^2}}{|K^2|}} \\
&= \sum_{K^1,K^2}\sum_{(I_1^1)^{(k_1)} = K^1} \Ave{b,h_{I_1^1} \otimes \frac{1_{K^2}}{|K^2|}} \ave{\ave{f_1, h_{K^2}}_2 }_{I_1^1}\sum_{\substack{(I_i^1)^{(k_i)} = K^1\\ i\neq1}} a_{K(I_i^1)} \prod_{i=2}^{n+1} \Ave{f_i,\wt h_{I_i^1} \otimes \frac{1_{K^2}}{|K^2|}}.
\end{align*}Note that we can rewrite the above as 
\begin{align*}
\sum_{J^1, J^2}\Ave{b,h_{J^1} \otimes \frac{1_{J^2}}{|J^2|}}\langle \sigma_1\rangle_{J^1\times J^2}\frac{\Ave{f_1, \frac{1_{J^1}}{|J^1|} \otimes h_{J^2}}}{\langle \sigma_1\rangle_{J^1\times J^2} }\langle (S\pi)_k(f_2, \cdots, f_n), h_{J^1}\otimes h_{J^2}\rangle.
\end{align*}
Then  there is nothing new here; by Lemma \ref{lem:MWestimate} we have that the above is dominated by
\begin{align*}
\|b\|_{\bmo(\nu)} &\Big\|\sum_{J^2} \Big(\sum_{J^1} \big[M_{\calD^1}^{\langle \sigma_1\rangle_{J^2,2}}(\langle f_1, h_{J^2}\rangle_2\langle \sigma_1\rangle_{J^2,2}^{-1} )\big]^2\\
&\hspace{3cm}\times\langle (S\pi)_k(f_2, \cdots, f_n), h_{J^1}\otimes h_{J^2}\rangle^2 \frac{1_{J^1}}{|J^1|} \Big)^{\frac 12} \frac{ 1_{J^2}}{|J^2|}\Big\|_{L^1(\sigma_1\nu)}.
\end{align*}The estimate is then completed by applying H\"older's inequality twice, Proposition \ref{prop:5.8} and Lemma \ref{lem:SFpartial}.

Symmetrically, we can work with $\Pi_{3,2}.$
Lastly, we focus on terms with $\Pi_{3,3}$ type illegal paraproducts. We choose here the type of term which we did not consider in the shift section:
$$\ave{(S\pi)_k(\Pi_{3,3}(b,f_1),\ldots,f_{n}) - b(S\pi)_k(f_1\ldots, f_n),f_{n+1}}.$$ Notice that we have
\begin{equation}\label{eq:errorTermPartial}
\begin{split}
&\ave{\Pi^{1,2}_{3,3}(b,f_1), h_{I_1^1} \otimes h_{K^2}} \Ave{f_{n+1} , h_{I_{n+1}^1}^0 \otimes \frac{1_{K^2}}{|K^2|}} - \ave{f_1, h_{I_1^1} \otimes h_{K^2}}\Ave{bf_{n+1} , h_{I_{n+1}^1}^0 \otimes \frac{1_{K^2}}{|K^2|}}\\
&=\ave{b}_{I_1^1\times K^2}\ave{f_1, h_{I_1^1} \otimes h_{K^2}}\Ave{f_{n+1} , h_{I_{n+1}^1}^0 \otimes \frac{1_{K^2}}{|K^2|}} - \ave{f_1, h_{I_1^1} \otimes h_{K^2}}\Ave{bf_{n+1} , h_{I_{n+1}^1}^0 \otimes \frac{1_{K^2}}{|K^2|}} \\
&=\Big(\ave{b}_{I_1^1\times K^2} - \ave{b}_{I_{n+1}^1\times K^2}\Big) \ave{f_1, h_{I_1^1} \otimes h_{K^2}}\Ave{f_{n+1} , h_{I_{n+1}^1}^0 \otimes \frac{1_{K^2}}{|K^2|}}\\
&\hspace{5cm} -\ave{f_1, h_{I_1^1} \otimes h_{K^2}}\Ave{(b - \ave{b}_{I_{n+1}^1\times K^2})f_{n+1} , h_{I_{n+1}^1}^0 \otimes \frac{1_{K^2}}{|K^2|}}.
\end{split}
\end{equation}
Now on the right-hand side of the above equation \eqref{eq:errorTermPartial}, we have two distinct cases where the first part is similar to the ones seen in the analysis of the shift commutator. We begin with this familiar case. However, now without using the sharper \eqref{eq:shiftM2} expansion since, in this case, it does not matter if we have a square root dependence or a linear one. Observe that
$$
|\ave{b}_{I_1^1\times K^2} - \ave{b}_{I_{n+1}^1\times K^2} | \lesssim \|b\|_{\bmo(\nu)} (\nu_{I_1^1,K^2}+\nu_{I_{n+1}^1,K^2}),
$$
where
$$
\nu_{Q_1,K^2} := \sum_{\substack{J^1 \in \calD^1\\ Q^1\subsetneq J^1 \subset K^1}} \ave{\nu}_{J^1 \times K^2},\quad Q^1\in \{I_1^1, I_{n+1}^1\}.
$$
Then our term is bounded by  
\begin{align*}
& \|b\|_{\bmo(\nu)}\sum_{K} \sum_{(I_i^1)^{(k_i)} = K^1} |a_{K,(I_i^1)}| |\ave{f_1, h_{I_1^1} \otimes h_{K^2}} |\prod_{i=2}^{n+1}\Big|\Ave{f_i, \wt h_{I_i^1} \otimes \frac{1_{K^2}}{|K^2|}} \Big|\\
&\hspace{9cm}\times(\nu_{I_1^1,K^2}+\nu_{I_{n+1}^1,K^2}).
\end{align*}
We first consider $\nu_{I_{n+1}^1,K^2}$. We fix $j_{n+1} \in \{1,\ldots, k_{n+1}\}$ and
it suffices to bound 
\begin{align*}
& \sum_{K^1}\sum_{(I_i^1)^{(k_i)} = K^1} \sum_{K^2} |a_{K,(I_{i}^1)}|\langle \nu\rangle_{(I_{n+1}^1)^{(j_{n+1})}\times K^2} |\ave{f_1, h_{I_1^1} \otimes h_{K^2}}| \prod_{i=2}^{n+1}\Big|\Ave{f_i, \wt h_{I_i^1} \otimes \frac{1_{K^2}}{|K^2|}}\Big|\\
&\lesssim \sum_{K^1}\sum_{(I_i^1)^{(k_i)} = K^1} \frac{\prod_{i=1}^{n+1}|I_i^1|^{\frac 12}}{|K^1|^n}\int_{\R^{d_1}}\frac{1_{(I_{n+1}^1)^{(j_{n+1})}}}{|(I_{n+1}^1)^{(j_{n+1})}|}\\
&\hspace{1cm}\times 
\int_{\R^{d_2}}\Big(\sum_{K^2} |\ave{f_1, h_{I_1^1} \otimes h_{K^2}}|^2 \prod_{i=2}^{n+1}\Big|\Ave{f_i, \wt h_{I_i^1} \otimes \frac{1_{K^2}}{|K^2|}}\Big|^2\frac{1_{K^2}}{|K^2|}\Big)^{\frac 12}\nu,
\end{align*}where we have applied Lemma \ref{lem:MW}.
Recall the strategy in \cite{LMV:gen}, when $\wt h_{I_i^1}=h_{I_i^1}$ we do not do anything and when $\wt h_{I_i^1}=h_{I_i^1}^0$ and $I_i^1\neq K^1$ we expand 
\[
|I^1_j|^{-\frac{1}{2}} \langle f_j, h^0_{I^1_j} \rangle_1
=\langle f\rangle_{I_j^1,1}
=\langle f_j \rangle_{K^1,1}+\sum_{i_j=1}^{k_j} \langle \Delta^1_{(I^1_j)^{(i_j)}} f_j\rangle_{(I_j^1)^{(i_j-1)},1}.
\]
We have 
\begin{align*}
&\sum_{K^1}\sum_{(I_i^1)^{(k_i)} = K^1} \frac{\prod_{i=1}^{n+1}|I_i^1|^{\frac 12}}{|K^1|^n}\int_{\R^{d_1}}\frac{1_{(I_{n+1}^1)^{(j_{n+1})}}}{|(I_{n+1}^1)^{(j_{n+1})}|} |I_{n+1}^1|^{\frac 12}\\
&\hspace{1cm}\times 
\int_{\R^{d_2}}\Big(\sum_{K^2} |\ave{f_1, h_{I_1^1} \otimes h_{K^2}}|^2 \prod_{i=2}^{n}\Big|\Ave{f_i, \wt h_{I_i^1} \otimes \frac{1_{K^2}}{|K^2|}}\Big|^2 \big|\langle f_{n+1}\rangle_{K^1\times K^2}\big|^2\frac{1_{K^2}}{|K^2|}\Big)^{\frac 12}\nu\\
&\le \sum_{K^1}\sum_{\substack{(I_i^1)^{(k_i)} = K^1\\ i\neq n+1}} \frac{\prod_{i=1}^{n}|I_i^1|^{\frac 12}}{|K^1|^n}\int_{\R^{d}}1_{K^1}\\
&\hspace{1cm}\times 
\Big(\sum_{K^2} |\ave{f_1, h_{I_1^1} \otimes h_{K^2}}|^2 \prod_{i=2}^{n}\Big|\Ave{f_i, \wt h_{I_i^1} \otimes \frac{1_{K^2}}{|K^2|}}\Big|^2 \langle |f_{n+1}|\rangle_{K^1\times K^2}^2\frac{1_{K^2}}{|K^2|}\Big)^{\frac 12}\nu.
\end{align*}
Since $(w_1, \cdots, w_n, \nu w^{-1})\in A_{(p_1, \cdots, p_n, p')}$, the same proof as in \cite{LMV:gen}*{Section 6.B} yields the desired estimate. The proof of $\langle \Delta^1_{(I^1_{n+1})^{(i_{n+1})}} f_{n+1}\rangle_{(I_{n+1}^1)^{(i_{n+1}-1)},1}$ with $i_{n+1}\ge j_{n+1}$ is similar. So we only focus on $i_{n+1}< j_{n+1}$. By simple calculus, we reduce to bounding 
\begin{align*}
&\sum_{K^1}\sum_{\substack{(I_i^1)^{(k_i)} = K^1\\ i\neq n+1}} \frac{\prod_{i=1}^{n}|I_i^1|^{\frac 12}}{|K^1|^n}\int_{\R^{d_1}}\frac{1_{(L_{n+1}^1)^{(j_{n+1}-i_{n+1})}}}{|(L_{n+1}^1)^{(j_{n+1}-i_{n+1})}|} |L_{n+1}^1|^{\frac 12}\\
&\hspace{0.5cm}\times 
\int_{\R^{d_2}}\Big(\sum_{K^2} |\ave{f_1, h_{I_1^1} \otimes h_{K^2}}|^2 \prod_{i=2}^{n}\Big|\Ave{f_i, \wt h_{I_i^1} \otimes \frac{1_{K^2}}{|K^2|}}\Big|^2 \Big|\Ave{f_{n+1}, h_{L_{n+1}^1}\otimes \frac{1_{K^2}}{|K^2|}}\Big|^2\frac{1_{K^2}}{|K^2|}\Big)^{\frac 12}\nu. 
\end{align*}
Denote $(L_{n+1}^1)^{(j_{n+1}-i_{n+1})}=Q_{n+1}^1$, and write 
\[
\Ave{f_{n+1}, h_{L_{n+1}^1}\otimes \frac{1_{K^2}}{|K^2|}}= \frac{\Ave{f_{n+1}, h_{L_{n+1}^1}\otimes \frac{1_{K^2}}{|K^2|}}}{\langle \sigma_{n+1}\rangle_{Q_{n+1}^1\times K^2}}\langle \sigma_{n+1}\rangle_{Q_{n+1}^1\times K^2}.
\]
By the reverse H\"older and $A_{\infty}$ extrapolation, we can get $\sigma_{n+1}$ out of the square sum. Then using 
\[
\frac{\Big|\Ave{f_{n+1}, h_{L_{n+1}^1}\otimes \frac{1_{K^2}}{|K^2|}}\Big|}{\langle \sigma_{n+1}\rangle_{Q_{n+1}^1\times K^2}}\le M_{\calD^2}^{\langle \sigma_{n+1}\rangle_{Q_{n+1}^1,1}}(\ave{f_{n+1}, h_{L_{n+1}^1}}_1\langle \sigma_{n+1}\rangle_{Q_{n+1}^1,1}^{-1})1_{K^2}
\]
we arrive at 
\begin{align*}
&\sum_{K^1}\sum_{\substack{(I_i^1)^{(k_i)} = K^1\\i\neq n+1}} \frac{\prod_{i=1}^{n}|I_i^1|^{\frac 12}}{|K^1|^n}\int_{\R^{d}}1_{K^1}F_{n+1,K^1}\\
&\hspace{3cm}\times 
\Big(\sum_{K^2} |\ave{f_1, h_{I_1^1} \otimes h_{K^2}}|^2 \prod_{i=2}^{n}\Big|\Ave{f_i, \wt h_{I_i^1} \otimes \frac{1_{K^2}}{|K^2|}}\Big|^2 \frac{1_{K^2}}{|K^2|}\Big)^{\frac 12}\nu\sigma_{n+1},
\end{align*}where
\begin{align*}
F_{n+1,K^1}&:=\sum_{(Q_{n+1}^1)^{(k_{n+1}-j_{n+1})}=K^1}\frac{1_{Q_{n+1}^1}}{|Q_{n+1}^1|}\sum_{(L_{n+1}^1)^{(j_{n+1}-i_{n+1})}=Q_{n+1}^1} |L_{n+1}^1|^{\frac 12}\\
&\hspace{3cm}\times 
M_{\calD^2}^{\langle \sigma_{n+1}\rangle_{Q_{n+1}^1,1}}(\ave{f_{n+1}, h_{L_{n+1}^1}}_1\langle \sigma_{n+1}\rangle_{Q_{n+1}^1,1}^{-1}).
\end{align*}By H\"older's inequality, it suffices to bound the $L^p(w^p)$ norm of 
\begin{align*}
\Big(\sum_{K^1}1_{K^1}\Big[\sum_{\substack{(I_i^1)^{(k_i)} = K^1\\i\neq n+1}} \frac{\prod_{i=1}^{n}|I_i^1|^{\frac 12}}{|K^1|^n}\Big(\sum_{K^2} |\ave{f_1, h_{I_1^1} \otimes h_{K^2}}|^2 \prod_{i=2}^{n}\Big|\Ave{f_i, \wt h_{I_i^1} \otimes \frac{1_{K^2}}{|K^2|}}\Big|^2 \frac{1_{K^2}}{|K^2|}\Big)^{\frac 12}\Big]^2\Big)^{\frac 12}
\end{align*}and 
\[
\Big\|\Big(\sum_{K^1} F_{n+1,K^1}^2\Big)^{\frac 12}\Big\|_{L^{p'}(\sigma_{n+1})}.
\]
Simply control the outer $\ell^2$ norm by $\ell^1$ norm -- then we can again use the estimate in \cite{LMV:gen}*{Section 6.B} to conclude the first term. For the second term, note that 
\begin{align*}
&\Big\|\Big(\sum_{K^1} F_{n+1,K^1}^2\Big)^{\frac 12}\Big\|_{L^{p'}(\sigma_{n+1})}\\
&=\Big\|\Big(\sum_{K^1}\Big[\sum_{(Q_{n+1}^1)^{(k_{n+1}-j_{n+1})}=K^1}\frac{1_{Q_{n+1}^1}}{|Q_{n+1}^1|} \sum_{(L_{n+1}^1)^{(j_{n+1}-i_{n+1})}=Q_{n+1}^1} |L_{n+1}^1|^{\frac 12}\\
&\hspace{4cm}\times 
M_{\calD^2}^{\langle \sigma_{n+1}\rangle_{Q_{n+1}^1,1}}(\ave{f_{n+1}, h_{L_{n+1}^1}}_1\langle \sigma_{n+1}\rangle_{Q_{n+1}^1,1}^{-1})\Big]^2\Big)^{\frac 12}\Big\|_{L^{p'}(\sigma_{n+1})}\\
&=\Big\| \Big(\sum_{Q_{n+1}^1} 1_{Q_{n+1}^1}\Big[\sum_{(L_{n+1}^1)^{(j_{n+1}-i_{n+1})}=Q_{n+1}^1} \frac{|L_{n+1}^1|^{\frac 12}}{|Q_{n+1}^1|}\\
&\hspace{4cm}\times M_{\calD^2}^{\langle \sigma_{n+1}\rangle_{Q_{n+1}^1,1}}(\ave{f_{n+1}, h_{L_{n+1}^1}}_1\langle \sigma_{n+1}\rangle_{Q_{n+1}^1,1}^{-1})\Big]^2\Big)^{\frac 12}\Big\|_{L^{p'}(\sigma_{n+1})},
\end{align*}which again can be handled exactly as in \cite{LMV:gen}*{p.23}. Now we turn to consider the case $Q^1= I_1^1$. Similarly, 
\begin{align*}
& \sum_{K^1}\sum_{(I_i^1)^{(k_i)} = K^1} \sum_{K^2} |a_{K,(I_{i}^1)}|\langle \nu\rangle_{(I_{1}^1)^{(j_{1})}\times K^2} |\ave{f_1, h_{I_1^1} \otimes h_{K^2}}| \prod_{i=2}^{n+1}\Big|\Ave{f_i, \wt h_{I_i^1} \otimes \frac{1_{K^2}}{|K^2|}}\Big|\\
&\lesssim \sum_{K^1}\sum_{(I_i^1)^{(k_i)} = K^1} \frac{\prod_{i=1}^{n+1}|I_i^1|^{\frac 12}}{|K^1|^n}\int_{\R^{d_1}}\frac{1_{(I_{1}^1)^{(j_{1})}}}{|(I_{1}^1)^{(j_{1})}|}\\
&\hspace{3.5cm}\times 
\int_{\R^{d_2}}\Big(\sum_{K^2} |\ave{f_1, h_{I_1^1} \otimes h_{K^2}}|^2 \prod_{i=2}^{n+1}\Big|\Ave{f_i, \wt h_{I_i^1} \otimes \frac{1_{K^2}}{|K^2|}}\Big|^2\frac{1_{K^2}}{|K^2|}\Big)^{\frac 12}\nu.
\end{align*}Similar as \cite{LMV:gen}, we may without loss of generality assume either $\wt h_{I_i^1}=h_{I_i^1} $ or otherwise $I_i^1=K^1$.  As before, by reverse H\"older and $A_\infty$ extrapolation the object is dominated by 
\begin{align*}
&\sum_{K^1}\sum_{(I_i^1)^{(k_i)} = K^1} \frac{\prod_{i=1}^{n+1}|I_i^1|^{\frac 12}}{|K^1|^n}\int_{\R^{d_1}}\frac{1_{(I_{1}^1)^{(j_{1})}}}{|(I_{1}^1)^{(j_{1})}|}\\
&\hspace{2.5cm}\times 
\int_{\R^{d_2}}\Big(\sum_{K^2} \frac{|\ave{f_1, h_{I_1^1} \otimes h_{K^2}}|^2}{\langle \sigma_1\rangle^2_{(I_1^1)^{(j_1)}\times K^2}} \prod_{i=2}^{n+1}\Big|\Ave{f_i, \wt h_{I_i^1} \otimes \frac{1_{K^2}}{|K^2|}}\Big|^2\frac{1_{K^2}}{|K^2|}\Big)^{\frac 12}\nu\sigma_1.
\end{align*}
Next, we write 
\begin{align*}
\prod_{i=2}^{n+1}\Big|\Ave{f_i, \wt h_{I_i^1} \otimes \frac{1_{K^2}}{|K^2|}}\Big|\le M_{\calD^2}(\langle f_2, \wt h_{I_2^1}\rangle_1, \cdots, \langle f_{n+1}, \wt h_{I_{n+1}^1}\rangle_1).
\end{align*}
Then by H\"older's inequality the estimate is reduced to 
\begin{align*}
A:=\Big\|\Big(\sum_{K^1}\Big(\sum_{(I_1^1)^{(k_1)}=K^1}|I_1^1|^{\frac 12}\frac{1_{(I_{1}^1)^{(j_{1})}}}{|(I_{1}^1)^{(j_{1})}|}\Big(\sum_{K^2} \frac{|\ave{f_1, h_{I_1^1} \otimes h_{K^2}}|^2}{\langle \sigma_1\rangle^2_{(I_1^1)^{(j_1)}\times K^2}} \frac{1_{K^2}}{|K^2|}\Big)^{\frac 12}\Big)^2\Big)^{\frac 12}\Big\|_{L^{p_1}(\sigma_1)}
\end{align*}
and 
\begin{align*}
&B:=\Big\|\Big(\sum_{K^1}1_{K^1}\Big[\sum_{\substack{(I_i^1)^{(k_i)} = K^1\\ i\neq 1}}\frac{\prod_{i=2}^{n+1}|I_i^1|^{\frac 12}}{|K^1|^n} M_{\calD^2}(\langle f_2, \wt h_{I_2^1}\rangle_1, \cdots, \langle f_{n+1}, \wt h_{I_{n+1}^1}\rangle_1)\Big]^2\Big)^{\frac 12}\Big\|_{L^{p_1'}(\eta_1)}.
\end{align*}
Again, the estimate of $A$ can be found in \cite{LMV:gen}*{Section 6.B} and we omit the details. For $B$, we shall prove 
\[
B\lesssim \|f_{n+1}\nu w^{-1} \|_{L^{p'}}\prod_{i=2}^{n} \|f_iw_i\|_{L^{p_i}}.
\]By the extrapolation theorem, it suffices to prove 
\begin{align*}
&\Big\|\Big(\sum_{K^1}1_{K^1}\Big[\sum_{\substack{(I_i^1)^{(k_i)} = K^1\\ i\neq 1}}\frac{\prod_{i=2}^{n+1}|I_i^1|^{\frac 12}}{|K^1|^n} M_{\calD^2}(\langle f_2, \wt h_{I_2^1}\rangle_1, \cdots, \langle f_{n+1}, \wt h_{I_{n+1}^1}\rangle_1)\Big]^2\Big)^{\frac 12}v\Big\|_{L^{\frac 2n}}\\
&\hspace{11.5cm} \le\prod_{i=2}^{n+1} \|f_iv_i\|_{L^{2}},
\end{align*}
provided $(v_2, \cdots, v_{n+1})\in A_{(2,\cdots, 2)}$ and $v=\prod_{i=2}^{n+1}v_i$. Note that for a fixed $K^2$, if we denote $\zeta_i=v_i^{-2}$, $2\le i\le n+1$, then
\begin{align*}
\prod_{i=2}^{n+1} \Ave{|\langle f_i, \wt h_{I_i^1} \rangle_1| }_{K^2} &=\prod_{i=2}^{n+1} \frac{\Ave{|\langle f_i, \wt h_{I_i^1} \rangle_1| }_{K^2}}{\langle \zeta_i\rangle_{K^1\times K^2}}\langle \zeta_i\rangle_{K^1\times K^2}\\
&\lesssim \frac 1{\langle v^{\frac 2n}\rangle_{K^1\times K^2}^n}\prod_{i=2}^{n+1} \frac{\Ave{|\langle f_i, \wt h_{I_i^1} \rangle_1| }_{K^2}}{\langle \zeta_i\rangle_{K^1\times K^2}}\\&\le \inf_{x\in K^1\times K^2}
\Big(M_{\calD}^{v^{\frac 2n}}\big(\big[\prod_{i=2}^{n+1}M_{\calD^2}^{\langle \zeta_i\rangle_{K^1,1}}( |\langle f_i, \wt h_{I_i^1} \rangle_1| \langle \zeta_i\rangle_{K^1,1}^{-1}) 1_{K^1}\big]^{\frac 1n}v^{-\frac 2n}\big)\Big)^n.
\end{align*}Whence
\begin{align*}
&1_{K^1}M_{\calD^2}(\langle f_2, \wt h_{I_2^1}\rangle_1, \cdots, \langle f_{n+1}, \wt h_{I_{n+1}^1}\rangle_1)\\
&\hspace{4cm}\le \Big(M_{\calD}^{v^{\frac 2n}}\big(\big[\prod_{i=2}^{n+1}M_{\calD^2}^{\langle \zeta_i\rangle_{K^1,1}} (|\langle f_i, \wt h_{I_i^1} \rangle_1| \langle \zeta_i\rangle_{K^1,1}^{-1}) 1_{K^1}\big]^{\frac 1n}v^{-\frac 2n}\big)\Big)^n
\end{align*}and by the vector-valued estimate for $M_{\calD}^{v^{\frac 2n}}$ and H\"older's inequality, we have 
\begin{align*}
&\Big\|\Big(\sum_{K^1}1_{K^1}\Big[\sum_{\substack{(I_i^1)^{(k_i)} = K^1\\ i\neq 1}}\frac{\prod_{i=2}^{n+1}|I_i^1|^{\frac 12}}{|K^1|^n} M_{\calD^2}(\langle f_2, \wt h_{I_2^1}\rangle_1, \cdots, \langle f_{n+1}, \wt h_{I_{n+1}^1}\rangle_1)\Big]^2\Big)^{\frac 12}v\Big\|_{L^{\frac 2n}}\\
&\hspace{2cm}\lesssim \prod_{i=2}^{n+1}\Big\|\Big(\sum_{K^1}1_{K^1}\Big[\sum_{(I_i^1)^{(k_i)} = K^1}\frac{|I_i|^{\frac 12}}{|K^1|}M_{\calD^2}^{\langle \zeta_i\rangle_{K^1,1}} (|\langle f_i, \wt h_{I_i^1} \rangle_1| \langle \zeta_i\rangle_{K^1,1}^{-1})\Big]^2\Big)^{\frac 12}\Big\|_{L^2(\zeta_i)}.
\end{align*}Recall that when $\wt h_{I_i^1}=h_{I_i^1}^0$, then according to our convention $I_i^1=K^1$ and 
\[
\sum_{(I_i^1)^{(k_i)} = K^1}\frac{|I_i|^{\frac 12}}{|K^1|}M_{\calD^2}^{\langle \zeta_i\rangle_{K^1,1}} (|\langle f_i, \wt h_{I_i^1} \rangle_1| \langle \zeta_i\rangle_{K^1,1}^{-1})= M_{\calD^2}^{\langle \zeta_i\rangle_{K^1,1}}(|\ave{f_i}_{K^1}|\langle \zeta_i\rangle_{K^1,1}^{-1})\le M_{\calD}^{\zeta_i}(f_i \zeta_i^{-1}).
\]
Again the rest can be estimated as in \cite{LMV:gen}*{Section 6.B}.

Next, we consider the latter part of \eqref{eq:errorTermPartial}. Notice that by Lemma \ref{lem:MWestimate} we have
\begin{align*}
&|I_{n+1}^1|^{\frac 12}|K^2|\Ave{(b - \ave{b}_{I_{n+1}^1\times K^2})f_{n+1} , h_{I_{n+1}^1}^0 \otimes \frac{1_{K^2}}{|K^2|}} \\&
=\sum_{I\times J\subset I_{n+1}^1\times K^2}\langle b, h_I\otimes h_J\rangle \langle f_{n+1}, h_I\otimes h_J\rangle+ \sum_{I\subset I_{n+1}^1} \Ave{ b, h_I \otimes \frac{1_{K^2}}{|K^2|}} \ave{f_{n+1}, h_I \otimes 1_{K^2}}\\
&\qquad+\sum_{J\subset K^2} \Ave{b, \frac{1_{I_{n+1}^1}}{|I_{n+1}^1|}\otimes h_J}\ave{f_{n+1},1_{I_{n+1}^1}\otimes h_J}\\
&\lesssim \|b\|_{\bmo(\nu)} \int_{I_{n+1}^1\times K^2}\Big(\sum_{R\in \calD} \frac{\ave{f_{n+1}, h_R}^2}{\langle \sigma_{n+1}\rangle_{R}^2}\frac{1_R}{|R|}\Big)^{\frac 12}\nu \sigma_{n+1}\\
&\hspace{1.5cm}+\|b\|_{\bmo(\nu)} \int_{I_{n+1}^1\times K^2}\Big(\sum_{I\in \calD^1} \frac{\ave{f_{n+1}, h_I\otimes \frac{1_{K^2}}{|K^2|}}^2}{\langle \sigma_{n+1}\rangle_{I\times K^2}^2}\frac{1_I}{|I|}\Big)^{\frac 12}\nu \sigma_{n+1}\\
&\hspace{3cm}+\|b\|_{\bmo(\nu)}\int_{I_{n+1}^1\times K^2}\Big(\sum_{J\in \calD^2} \frac{\ave{f_{n+1}, \frac{1_{I_{n+1}^1}}{|I_{n+1}^1|}\otimes h_J}^2}{\langle \sigma_{n+1}\rangle_{I_{n+1}^1\times J^2}^2}\frac{1_J}{|J|}\Big)^{\frac 12}\nu \sigma_{n+1}.
\end{align*}
Then e.g. dominating
\[
\Big(\sum_{I\in \calD^1} \frac{\ave{f_{n+1}, h_I\otimes \frac{1_{K^2}}{|K^2|}}^2}{\langle \sigma_{n+1}\rangle_{I\times K^2}^2}\frac{1_I}{|I|}\Big)^{\frac 12}\le \Big(\sum_{I\in \calD^1} \big[ M_{\calD^2}^{\langle \sigma_{n+1}\rangle_{I,1}} (\ave{f_{n+1}, h_I} \langle \sigma_{n+1}\rangle_{I,1}^{-1})\big]^2\frac{1_I}{|I|}\Big)^{\frac 12}
\]allows us to view these square functions (which are bounded on $L^{p'}(\sigma_{n+1})$) as the new $f_{n+1}$. So that by H\"older's inequality, 
the related term in the commutator boils down 
to estimating the partial paraproduct
\begin{align*}
\Big\| \sum_{K^1, K^2} \sum_{(I_j^1)^{(k_j)} = K^1} a_{K(I_i^1)}\ave{f_1, h_{I_1^1} \otimes h_{K^2}} \prod_{i=2}^n \Ave{f_i, h_{I_i^1} \otimes \frac{1_{K^2}}{|K^2|}}h_{I_{n+1}^1}^0 \otimes \frac{1_{K^2}}{|K^2|} w\Big\|_{L^p},
\end{align*}
which is exactly the standard one.

Following the expansion methods and estimations introduced earlier, we can handle the other forms of commutators similarly. Compared to the shift case, the more difficult challenges arise from the terms of forms, where we have
\begin{align*}
(\ave{b}_{I_1^1 \times K^2} - \ave{b}_{K})\ave{ f_1, h_{I_1^1} \otimes h_{K^2}}, \\
(\ave{b}_{I_1^1 \times K^2} - \ave{b}_{K})\Ave{ f_1, h_{I_1^1} \otimes \frac{1_{K^2}}{|K^2|}},\\
(\ave{b}_{I_1^1 \times K^2} - \ave{b}_{K})\ave{ f_1, h_{I_1^1}^0 \otimes h_{K^2}},
\end{align*}
and $$
(\ave{b}_{I_1^1 \times K^2} - \ave{b}_{K})\Ave{ f_1, h_{I_1^1}^0 \otimes \frac{1_{K^2}}{|K^2|}}.$$
We already handled the first and the symmetric case of the last one.
By modifying the above methods, we can estimate the other two terms.


\subsection*{Full paraproducts}
Although the full paraproducts have the more complicated product BMO coefficients, they do not require as much analysis as the partial paraproducts. Since no unseen methods are needed to conclude the boundedness of full paraproduct commutators, we omit the details.

\section{The lower bound}\label{sec:lower}
Let $K$ be a standard bi-parameter full kernel as described earlier. In this section, we additionally assume that $K$ is a multilinear non-degenerate kernel. 
That is, for any given rectangle $R = I^1 \times I^2$ there exists $\wt R = \wt I^1 \times \wt I^2$ such that $\ell(I^i) = \ell(\wt I^i),$ $d(I^i, \wt I^i) \sim \ell(I^i),$ and there exists some $\zeta \in \C$ with $|\zeta|=1$ such that for all $x \in \wt R$ and $y_1,\ldots, y_n \in R$ there holds  
$$
\Re \zeta K(x, y_1,\ldots, y_n) \gtrsim \frac{1}{|R|^{n}}.
$$

We are going to assume the weak type boundedness of the commutator. Suppose that 
$$
\sup_{A \subset R} \frac{1}{\prod_{i=1}^n \sigma_i (R)^{\frac{1}{p_i}}} \Big \| 1_{\wt R}[b,T]_{j} (1_R \sigma_1, \ldots, 1_A \sigma_j, \ldots, 1_R \sigma_n) \nu^{-1} w\Big\|_{L^{p,\infty}} < \infty,
$$
where recall that $\sigma_i = w_i^{-p_i'}$ and $\nu = \lambda_j^{-1} w_j.$
Clearly, this is a weaker assumption than
$$
 \Big\|[b,T]_j \colon\prod_{i=1}^n L^{p_i}(w_i^{p_i}) \to L^p( \nu^{-p}w^p)\Big\| < \infty.
$$

We do not assume the two separate $A_{\vec p}$ conditions here. It is enough to assume that we have the two tuples $(w_1,\ldots,w_n), (w_1, \ldots,\lambda_j,\ldots, w_n)$ of weights satisfying $$(w_1, \ldots, w_n, \nu w^{-1}) \in A^*_{\vec p}.$$  Let us denote $\nu^{-p} w^{p}$ by $\sigma_{n+1}.$

We employ the idea of the median method to prove that $$b \in \bmo_\nu(\sigma_j) := \{b \in L^1_{\loc}\colon\sup_{R \in \calD} \inf_{c\in \R}\frac{1}{ \nu \sigma_j(R)} \int_R |b -c|\sigma_j < \infty\}$$
under the weaker assumption above. We additionally need to assume that $\nu \sigma_j \in A_\infty$ since when $\nu,\sigma_j,\nu\sigma_j \in A_\infty$ it follows that this is equivalent with the Bloom type little BMO definition, see Proposition \ref{prop:bloomlittlebmo}. 

\begin{rem}
 We get $\nu \sigma_j \in A_\infty$ for free whenever $\lambda_j^{-p_j'} \in A_\infty$ since
 $$
 \nu \sigma_j = \lambda_j^{-1} w_j^{1-p_j'} = (\lambda_j^{-p_j'})^{\frac{1}{p_j'}}(\sigma_j)^{\frac{1}{p_j}} \in A_\infty.
 $$
\end{rem}

 Fix rectangle $R \in \calD.$ We take arbitrary $\alpha \in \R$ and $x\in \wt R \cap \{ b \ge \alpha \},$ where $\wt R$ is a rectangle that satisfies the non-degeneracy property. Thus, we have
\begin{align*}
&\frac{1}{|R|} \int_R (\alpha - b)_+ \sigma_j \prod_{\substack{i =1\\i\neq j}}^{n} \frac{\sigma_i(R)}{|R|} \\ 
&\lesssim \Re \zeta \int_{R \cap \{ b \le \alpha\}} \int_{R}\ldots\int_{R} (b(x) - b(y_j)) K(x, y_1,\ldots,y_n) \prod_{i=1}^n \sigma_i(y_i) \ud y_i.
\end{align*}

We let $\alpha$ be the median of $b$ on $\wt R,$ i.e.
$$
\min(|\wt R \cap \{b \le \alpha\}|, | \wt R \cap \{ b \ge \alpha \}|) \ge \frac{|\wt R|}{2} =  \frac{|R|}{2}.
$$
As $\sigma_{n+1}  \in A_\infty$ we have that $\sigma_{n+1} (\wt R \cap \{ b \ge \alpha \}) \sim \sigma_{n+1}(\wt R) \sim \sigma_{n+1} (R).$

Thus, we get 
\begin{equation}\label{eq:lowerEst}
\sigma_{n+1} (R)^{\frac1p} \frac{1}{|R|} \int_R (\alpha - b)_+ \sigma_j \prod_{\substack{i =1\\i\neq j}}^{n} \frac{\sigma_i(R)}{|R|} \lesssim \|C_b^K(\sigma_1,\ldots,\sigma_n) \|_{L^{p,\infty}(\sigma_{n+1})} \lesssim \prod_{i = 1}^n \sigma_i(R)^{\frac{1}{p_i}},
\end{equation}
where
\begin{align*}
&C_b^K(\sigma_1,\ldots,\sigma_n)(x) \\
&:=  1_{\wt R \cap \{b \ge \alpha\}}(x)\Re \zeta \int_{R \cap \{ b \le \alpha\}} \int_{R}\ldots\int_{R} (b(x) - b(y_j)) K(x, y_1,\ldots,y_n) \prod_{i=1}^n \sigma_i(y_i) \ud y_i.
\end{align*}

Recall that 
\begin{align*}
1 \leq \prod_{i=1}^n \ave{\sigma_i}_R^{\frac{1}{p_i'}} \ave{\sigma_{n+1}}_R^{\frac 1p} \ave{\nu}_R \leq [(w_1,\ldots,w_n,\nu w^{-1})]_{A^*_{\vec p}} < \infty.
\end{align*}
Rearranging terms in \eqref{eq:lowerEst} and using the observation, we get 
$$
\frac{1}{\ave{\nu}_R \sigma_j(R)  } \int_R (\alpha- b)_+ \sigma_j \lesssim 1.
$$
By the reverse H\"older property, we have
$$
\frac{1}{\ave{\nu}_R \sigma_j(R)  } \gtrsim \frac{1}{ \nu\sigma_j(R) }. 
$$

By symmetrical estimates, we also get 
$$
\frac{1}{\nu \sigma_j(R)  } \int_R (b - \alpha)_+ \sigma_j \lesssim 1.
$$
This completes the proof.

\section{Two-weight extrapolation}\label{sec:extrapo}
This section is devoted to proving 
Theorem \ref{thm:extrapo}.

The strategy of the proof will be similar as in \cite{LMO} and \cite{LMMOV}. We only prove the case 
\[
q_n\neq p_n, 1<q_n\le \infty, q_i=p_i\, \text{for all $2\le i\le n-1$}.
\] 
Let us first recall the following lemma, whose proof can be found in 
 \cite{LMMOV}*{Lemma 2.14}.
 \begin{lem}
 Let $w_i^{\frac{1}{n-1+\frac 1{p_i}}}\in A_{\frac n{n-1+\frac 1{p_i}}}$, $1\le i\le n-1$. Let $\widehat w=(\prod_{i=1}^{n-1} w_i)^{\rho}\in A_{n\rho}$, where $\rho=(1+\sum_{i=1}^{n-1}\frac 1{p_i})^{-1}$. Then $(w_1,\cdots, w_n)\in A_{\vec p}$ if and only if 
 \[
 W:= w_n \widehat w^{\frac 1{p_n'}} \in A_{p_n, p}(\widehat w).
 \]
 \end{lem} 
 Note that it is also recorded in \cite{LMMOV}*{Lemma 2.14} that if $(w_1,\cdots, w_n)\in A_{\vec p}$, then we always have 
 \[
 \widehat w=(\prod_{i=1}^{n-1} w_i)^{\rho}\in A_{n\rho},\qquad w_i^{\frac{1}{n-1+\frac 1{p_i}}}\in A_{\frac n{n-1+\frac 1{p_i}}}, \quad i=1,\cdots, n-1.
 \]
With this at hand, since we have 
\[
(w_1, w_2,\cdots, w_n)\in A_{(p_1, \cdots, p_{n-1}, q_n)}, \quad (\lambda_1, w_2,\cdots, w_n)\in A_{(p_1, \cdots, p_{n-1}, q_n)},
\]
recalling that
\[
\frac 1q=\frac 1{p_1}+\cdots +\frac 1{p_{n-1}}+\frac 1{q_n},
\]we have 
\[
\widehat w=(\prod_{i=1}^{n-1} w_i)^{\rho}\in A_{n\rho},\quad \widehat \lambda=(\lambda_1\prod_{i=2}^{n-1} w_i)^{\rho}\in A_{n\rho}
\]and 
\[
W_w= w_n \widehat w^{\frac 1{q_n'}} \in A_{q_n, q}(\widehat w),\quad W_\lambda= w_n \widehat \lambda^{\frac 1{q_n'}} \in A_{q_n, q}(\widehat \lambda).
\]
Then the goal is to prove 
\[
\|f \lambda_1 w_1^{-1} W_w\|_{L^q(\widehat w)}\lesssim \|f_n \widehat w^{-1} W_w\|_{L^{q_n}(\widehat w)}\prod_{i=1}^{n-1} \|f_iw_i\|_{L^{p_i}},
\]
which can also be written as
\[
\|fW_\lambda\|_{L^q(\widehat \lambda)}\lesssim \|f_n \widehat {\lambda}^{-1} W_\lambda\|_{L^{q_n}(\widehat \lambda)}\prod_{i=1}^{n-1} \|f_iw_i\|_{L^{p_i}}.
\]
We split the proof to the following cases:

\textbf{Case 1: $1/s:=1/q-1/p=1/{q_n}-1/{p_n}>0$}. Without loss of generality we may assume 
\[
0<\|f_nw_n\|_{L^{q_n}}=\|f_n \widehat w^{-1} W_w\|_{L^{q_n}(\widehat w)}= \|f_n \widehat {\lambda}^{-1} W_\lambda\|_{L^{q_n}(\widehat \lambda)}<\infty.
\]
Let 
\[
h= \frac{f_n w_n^{q_n'}}{\| f_nw_n\|_{L^{q_n}}}=\frac{ f_n \widehat w^{-1} W_w^{q_n'} }{\|f_n \widehat w^{-1} W_w\|_{L^{q_n}(\widehat w)}}=  \frac{ f_n \widehat \lambda^{-1} W_\lambda^{q_n'} }{\|f_n \widehat \lambda^{-1} W_\lambda\|_{L^{q_n}(\widehat \lambda)}},
\]so that we have $\|h\|_{L^{q_n}(w_n^{-q_n'})}=1$.
Define
\[
\mathcal R' h= \sum_{k=0}^\infty \frac{(M_{\widehat \lambda}'M_{\widehat w}')^{(k)}h}{2^k \| M_{\widehat \lambda}'M_{\widehat w}'\|_{L^{(1+\frac{q_n'}{q})' }(w_n^{-q_n'})}^k}=: \sum_{k=0}^\infty \frac{(M_{\widehat \lambda}'M_{\widehat w}')^{(k)}h}{2^k \| M_{\widehat \lambda}'M_{\widehat w}'\|^k},
\]
where 
\[
M_{\widehat w}' g= M_{\widehat w}(g W_w^{-q_n'})W_w^{q_n'},\qquad M_{\widehat \lambda}' g= M_{\widehat \lambda}(g W_\lambda^{-q_n'})W_\lambda^{q_n'}.
\]Let us explain why $\mathcal R'$ is well-defined. Indeed,
since $W_w\in A_{q_n, q}(\widehat w)$, we have $W_w^{-q_n'}\in A_{1+\frac{q_n'}{q}}(\widehat w)$ and $M_{\widehat w}'$ is bounded on $L^{(1+\frac{q_n'}{q})' }(W_w^{-q_n'} \widehat w)=L^{(1+\frac{q_n'}{q})' }(w_n^{-q_n'})$ (see \cite{LMV:gen}*{Lemma 8.2}). Likewise $M_{\widehat \lambda}'$ is bounded on $L^{(1+\frac{q_n'}{q})' }(w_n^{-q_n'})$. 
Now set 
\[
H=\mathcal R'(h^{\frac{q_n}{(1+\frac{q_n'}{q})'}})^{\frac{(1+\frac{q_n'}{q})'}{q_n}}.
\] Then the above discussion easily yields 
\begin{align*}
h\le H, \quad \| H\|_{L^{q_n}(w_n^{-q_n'})}\lesssim \| h\|_{L^{q_n}(w_n^{-q_n'})}=1
\end{align*}
and 
\begin{align*}
M_{\widehat w}' \Big( H^{\frac{q_n}{(1+\frac{q_n'}{q})'}}\Big)&\le M_{\widehat \lambda}'M_{\widehat w}'\Big( H^{\frac{q_n}{(1+\frac{q_n'}{q})'}}\Big)\le 2\| M_{\widehat \lambda}'M_{\widehat w}'\| H^{\frac{q_n}{(1+\frac{q_n'}{q})'}}\\
M_{\widehat \lambda}' \Big( H^{\frac{q_n}{(1+\frac{q_n'}{q})'}}\Big)&\le M_{\widehat \lambda}'M_{\widehat w}'\Big( H^{\frac{q_n}{(1+\frac{q_n'}{q})'}}\Big)\le 2\| M_{\widehat \lambda}'M_{\widehat w}'\| H^{\frac{q_n}{(1+\frac{q_n'}{q})'}},
\end{align*}which give that
\[
[H^{\frac{q_n}{(1+\frac{q_n'}{q})'}} W_w^{-q_n'}]_{A_1(\widehat w)}\le 2\| M_{\widehat \lambda}'M_{\widehat w}'\| \quad \text{and}\quad  [H^{\frac{q_n}{(1+\frac{q_n'}{q})'}} W_\lambda^{-q_n'}]_{A_1(\widehat \lambda)}\le 2\| M_{\widehat \lambda}'M_{\widehat w}'\| .
\]
Finally, set 
$
v_n
=  H^{-\frac {q_n}s}w_n^{1+\frac{q_n'}s}.
$
It remains to check 
\begin{equation}\label{eq:extra-11}
(w_1, \cdots, w_{n-1}, v_n), (\lambda_1, \cdots, w_{n-1}, v_n)\in A_{\vec p}.
\end{equation}
Equivalently, we check 
\[
v_n \widehat w^{\frac 1{p_n'}}\in A_{p_n, p}(\widehat w)\quad \text{and}\quad v_n \widehat \lambda^{\frac 1{p_n'}}\in A_{p_n, p}(\widehat \lambda),
\]which will be completely similar as that in \cite{LMMOV}*{p. 106}. Indeed, once we have \eqref{eq:extra-11}, then 
\begin{align*}
\|fW_\lambda\|_{L^q(\widehat \lambda)}&=\|f v_n \widehat \lambda^{\frac 1q+\frac 1{q_n'}}   H^{\frac {q_n}s}w_n^{-\frac{q_n'}s}\|_{L^q}\le \|f v_n \widehat \lambda^{\frac 1q+\frac 1{q_n'}} \|_{L^p} \| H^{\frac {q_n}s}w_n^{-\frac{q_n'}s}\|_{L^s}\\
&\lesssim \|f v_n \widehat \lambda^{\frac 1q+\frac 1{q_n'}} \|_{L^p} =\|f v_n \widehat \lambda^{\frac 1p+\frac 1{p_n'}} \|_{L^p}\lesssim \| f_n v_n\|_{L^{p_n}} \prod_{i=1}^{n-1} \|f_i w_i\|_{L^{p_i}}.
\end{align*}The proof is completed by noticing that 
\begin{align*}
\| f_n v_n\|_{L^{p_n}} &=\big\| h w_n^{-q_n'} \|f_n w_n\|_{L^{q_n}} v_n\big\|_{L^{p_n}} \le \|f_n w_n\|_{L^{q_n}}  \| H^{1-\frac {q_n}s} w_n^{-q_n'+1+\frac{q_n'}s}\|_{L^{p_n}}\\
&= \|f_n w_n\|_{L^{q_n}} \| H^{\frac {q_n}{p_n}} w_n^{-\frac{q_n'}{p_n}}\|_{L^{p_n}}\lesssim \|f_n w_n\|_{L^{q_n}}.
\end{align*}

\textbf{Case 2: $1/s:=1/p-1/q=1/{p_n}-1/{q_n}>0$}.  Note that this case allows $q_n=\infty$. As observed in the above, $W_w^{-q_n'}\in A_{1+\frac{q_n'}{q}}(\widehat w)$ and thus $M_{\widehat w}$ is bounded on $L^{1+\frac{q_n'}{q} }(W_w^{-q_n'} \widehat w)=L^{1+\frac{q_n'}{q} }(w_n^{-q_n'})$. Likewise, $M_{\widehat \lambda}$ is bounded on $L^{1+\frac{q_n'}{q} }(w_n^{-q_n'})$. Denote by $\|M_{\widehat \lambda} M_{\widehat w}\|$ the norm of $M_{\widehat \lambda} M_{\widehat w}$ on $L^{1+\frac{q_n'}{q} }(w_n^{-q_n'})$. We introduce the following Rubio de Francia algorithm:
\[
\mathcal R g=\sum_{k=0}^\infty \frac{(M_{\widehat \lambda}M_{\widehat w})^{(k)}g}{2^k \| M_{\widehat \lambda}M_{\widehat w}\|^k}.
\]
By duality, there exists some $0\le h \in L^{\frac sp}( W_\lambda^q \widehat \lambda )$ such that $\|h\|_{L^{\frac sp}( W_\lambda^q \widehat \lambda )}=1$ and 
\[
\|fW_\lambda\|_{L^q(\widehat \lambda)}=\| f^p  \|_{L^{\frac qp}(W_\lambda^q \widehat \lambda )}^{\frac 1p}=\Big(\int f^p h W_\lambda^q \widehat \lambda\Big)^{\frac 1p}.
\]
Set 
\[
H=\mathcal R\Big (h^{\frac s{p(1+\frac{q_n'}q)}} w_n^{\frac{q_n'}{1+\frac{q_n'}q}} (W_\lambda^q \widehat \lambda)^{\frac 1{1+\frac{q_n'}q}} \Big)^{\frac{p(1+\frac{q_n'}q)}{s}} w_n^{-\frac{q_n'p}s}(W_\lambda^q \widehat \lambda)^{-\frac ps}.
\]
Then it is easy to check that 
\[
h\le H, \quad \| H\|_{L^{\frac sp}( W_\lambda^q \widehat \lambda )}\lesssim \| h\|_{L^{\frac sp}( W_\lambda^q \widehat \lambda )}=1
\]and 
\begin{align*}
\Big[w_n^{\frac{q_n'}{1+\frac{q_n'}{q}}} (W_\lambda^q \widehat \lambda)^{\frac 1{1+\frac{q_n'}q}} H^{\frac s{p(1+\frac{q_n'}q)}}\Big]_{A_1(\widehat w)}&\le 2\|M_{\widehat \lambda} M_{\widehat w}\|;\\
\Big[w_n^{\frac{q_n'}{1+\frac{q_n'}{q}}} (W_\lambda^q \widehat \lambda)^{\frac 1{1+\frac{q_n'}q}} H^{\frac s{p(1+\frac{q_n'}q)}}\Big]_{A_1(\widehat \lambda)}&\le 2\|M_{\widehat \lambda} M_{\widehat w}\|.
\end{align*}
Denote $v_n=H^{\frac 1p}W_\lambda^{\frac qp}\widehat \lambda^{-\frac 1{p_n'}}$,  we claim 
\begin{equation}\label{eq:extra-22}
(w_1, \cdots, w_{n-1}, v_n), (\lambda_1, \cdots, w_{n-1}, v_n)\in A_{\vec p}.
\end{equation}Assume \eqref{eq:extra-22} for the moment, then
\begin{align*}
\|fW_\lambda\|_{L^q(\widehat \lambda)}&= \Big(\int f^p h W_\lambda^q \widehat \lambda\Big)^{\frac 1p}\le \| f v_n \widehat \lambda^{\frac 1{p_n'}}\|_{L^p(\widehat\lambda)}\lesssim \|f_n v_n\|_{L^{p_n}} \prod_{i=1}^{n-1} \|f_i w_i\|_{L^{p_i}}.
\end{align*}
We can conclude this case by noticing that 
\begin{align*}
\|f_n v_n\|_{L^{p_n}}&\le \|f_n w_n\|_{L^{q_n}}\|v_n w_n^{-1}\|_{L^s}=\|f_n w_n\|_{L^{q_n}} \| H^{\frac 1p}W_\lambda^{\frac qp}\widehat \lambda^{-\frac 1{p_n'}}w_n^{-1}\|_{L^s}\\
&=\|f_n w_n\|_{L^{q_n}} \| H^{\frac 1p}W_\lambda^{\frac qs}\widehat \lambda^{\frac 1s} \|_{L^s}\lesssim \|f_n w_n\|_{L^{q_n}}.
\end{align*}
It remains to prove \eqref{eq:extra-22}. Similar as before, it suffices to prove 
\[
v_n \widehat w^{\frac 1{p_n'}}\in A_{p_n, p}(\widehat w)\quad \text{and}\quad v_n \widehat \lambda^{\frac 1{p_n'}}\in A_{p_n, p}(\widehat \lambda).
\]
Since 
\[
1-\frac {p(1+\frac{q_n'}q)}s=\frac{pq_n'}{qp_n'},
\]for arbitrary rectangle $Q$, direct calculus gives us 
\begin{align*}
&\Big(\frac 1{\widehat w(Q)}\int_Q v_n^p \widehat w^{\frac p{p_n'}}\widehat w\Big)^{\frac 1p}
=\Big(\frac 1{\widehat w(Q)}\int_Q H W_\lambda^q \widehat \lambda^{-\frac p{p_n'}} \widehat w^{\frac p{p_n'}+1}\Big)^{\frac 1p}\\
&= \Big(\frac 1{\widehat w(Q)}\int_Q H ( W_\lambda^q \widehat \lambda)^{\frac ps}w_n^{\frac {pq_n'}s}W_\lambda^{\frac {pq_n'}{p_n'}} \widehat \lambda^{-\frac p{p_n'}} \widehat w^{\frac p{p_n'}+1}\Big)^{\frac 1p}\\
&\le \Big(\frac 1{\widehat w(Q)}\int_Q w_n^{\frac{q_n'}{1+\frac{q_n'}{q}}} (W_\lambda^q \widehat \lambda)^{\frac 1{1+\frac{q_n'}q}} H^{\frac s{p(1+\frac{q_n'}q)}}\widehat w\Big)^{\frac {1+\frac{q_n'}q}s} \Big(\frac 1{\widehat w(Q)}\int_Q W_\lambda^{q} \widehat \lambda^{-\frac q{q_n'}} \widehat w^{\frac q{q_n'}+1}\Big)^{\frac {q_n'}{qp_n'}}\\
&\lesssim  \inf_Q  \big(w_n^{\frac {q_n'}s} (W_\lambda^q \widehat \lambda)^{\frac 1s}H^{\frac 1p}\big)\Big(\frac 1{\widehat w(Q)}\int_Q W_w^q \widehat w\Big)^{\frac {q_n'}{qp_n'}}.
\end{align*}
Thus 
\begin{align*}
&\Big(\frac 1{\widehat w(Q)}\int_Q v_n^p \widehat w^{\frac p{p_n'}}\widehat w\Big)^{\frac {q_n'}{qp_n'}}\Big(\frac 1{\widehat w(Q)}\int_Q v_n^{-p_n'}\Big)^{\frac 1{p_n'}}\\
&\lesssim \Big(\frac 1{\widehat w(Q)}\int_Q W_w^q \widehat w\Big)^{\frac {q_n'}{qp_n'}}\Big(\frac 1{\widehat w(Q)}\int_Q w_n^{\frac {p_n'q_n'}s} (W_\lambda^q \widehat \lambda)^{\frac {p_n'}s}H^{\frac {p_n'}p}v_n^{-p_n'}\Big)^{\frac 1{p_n'}}\\
&=\Big(\frac 1{\widehat w(Q)}\int_Q W_w^q \widehat w\Big)^{\frac {q_n'}{qp_n'}}\Big(\frac 1{\widehat w(Q)}\int_Q  w_n^{-q_n'}\Big)^{\frac 1{p_n'}}\le [W_w]_{A_{q_n,q}(\widehat w)}^{\frac {q_n'}{p_n'}}
\end{align*}This proves $v_n \widehat w^{\frac 1{p_n'}}\in A_{p_n, p}(\widehat w)$. The proof of $v_n \widehat \lambda^{\frac 1{p_n'}}\in A_{p_n, p}(\widehat \lambda)$ is similar. 
\bibliography{references}

\begin{bibdiv}
\begin{biblist}

\bib{EA}{article}{
      author={Airta, E.},
       title={Two-weight commutator estimates: general multi-parameter
  framework},
        date={2020},
        ISSN={0214-1493},
     journal={Publ. Mat.},
      volume={64},
      number={2},
       pages={681\ndash 729},
         url={https://doi.org/10.5565/PUBLMAT6422013},
}

\bib{ALMV}{article}{
      author={Airta, E.},
      author={Li, K.},
      author={Martikainen, H.},
      author={Vuorinen, E.},
       title={Some new weighted estimates on product space},
        date={2019},
     journal={Indiana Univ. Math. J.},
      eprint={https://www.iumj.indiana.edu/IUMJ/Preprints/8807.pdf},
      status={to appear},
}

\bib{AMV}{article}{
      author={Airta, E.},
      author={Martikainen, H.},
      author={Vuorinen, E.},
       title={Modern singular integral theory with mild kernel regularity},
        date={2019},
      eprint={https://arxiv.org/abs/2006.05807},
      status={preprint},
}

\bib{CRW}{article}{
      author={Coifman, R.~R.},
      author={Rochberg, R.},
      author={Weiss, G.},
       title={Factorization theorems for {H}ardy spaces in several variables},
        date={1976},
     journal={Ann. of Math. (2)},
      volume={103},
      number={3},
       pages={611\ndash 635},
}

\bib{CUMP}{article}{
      author={Cruz-Uribe, D.},
      author={Martell, J.~M.},
      author={P\'{e}rez, C.},
       title={Extrapolation from {$A_\infty$} weights and applications},
        date={2004},
     journal={J. Funct. Anal.},
      volume={213},
      number={2},
       pages={412\ndash 439},
}

\bib{DU}{article}{
      author={Duoandikoetxea, J.},
       title={Extrapolation of weights revisited: new proofs and sharp bounds},
        date={2011},
        ISSN={0022-1236},
     journal={J. Funct. Anal.},
      volume={260},
      number={6},
       pages={1886\ndash 1901},
         url={https://doi.org/10.1016/j.jfa.2010.12.015},
}

\bib{RF3}{article}{
      author={Fefferman, R.},
       title={Strong differentiation with respect to measures},
        date={1981},
     journal={Amer. J. Math.},
      volume={103},
       pages={33\ndash 40},
}

\bib{GLTP}{article}{
      author={Grafakos, L.},
      author={Liu, L.},
      author={Torres, R.~H.},
      author={P\'erez, C.},
       title={The multilinear strong maximal function},
        date={2011},
     journal={J. Geom. Anal.},
      volume={2},
       pages={118\ndash 149},
}

\bib{GM}{article}{
      author={Grafakos, L.},
      author={Martell, J.~M.},
       title={Extrapolation of weighted norm inequalities for multivariable
  operators and applications},
        date={2004},
        ISSN={1050-6926},
     journal={J. Geom. Anal.},
      volume={14},
      number={1},
       pages={19\ndash 46},
         url={https://doi.org/10.1007/BF02921864},
}

\bib{HPW}{article}{
      author={Holmes, I.},
      author={Petermichl, S.},
      author={Wick, B.~D.},
       title={Weighted little bmo and two-weight inequalities for {J}ourn\'e
  commutators},
        date={2018},
     journal={Anal. PDE},
      volume={11},
      number={7},
       pages={1693\ndash 1740},
}

\bib{HyCom}{article}{
      author={Hyt\"{o}nen, T.},
       title={The ${L}^p$-to-${L}^q$ boundedness of commutators with
  applications to the {J}acobian operator},
        date={2018},
      eprint={https://arxiv.org/abs/1804.11167},
      status={preprint},
}

\bib{Kunwar2018}{article}{
      author={Kunwar, I.},
      author={Ou, Y.},
       title={Two-weight inequalities for multilinear commutators},
        date={2018},
     journal={New York J. Math.},
      volume={24},
       pages={980\ndash 1003},
}

\bib{LOPTT}{article}{
      author={Lerner, A.~K.},
      author={Ombrosi, S.},
      author={P\'{e}rez, C.},
      author={Torres, R.~H.},
      author={Trujillo-Gonz\'{a}lez, R.},
       title={New maximal functions and multiple weights for the multilinear
  {C}alder\'{o}n-{Z}ygmund theory},
        date={2009},
     journal={Adv. Math.},
      volume={220},
      number={4},
       pages={1222\ndash 1264},
}

\bib{LOR1}{article}{
      author={Lerner, A.~K.},
      author={Ombrosi, S.},
      author={Rivera-R\'{i}os, I.~P.},
       title={On pointwise and weighted estimates for commutators of
  {C}alder\'{o}n-{Z}ygmund operators},
        date={2017},
     journal={Adv. Math.},
      volume={319},
       pages={153\ndash 181},
}

\bib{LOR2}{article}{
      author={Lerner, A.~K.},
      author={Ombrosi, S.},
      author={Rivera-R\'{\i}os, I.~P.},
       title={Commutators of singular integrals revisited},
        date={2019},
     journal={Bull. Lond. Math. Soc.},
      volume={51},
      number={1},
       pages={107\ndash 119},
}

\bib{LORR:twoweight}{article}{
      author={Lerner, A.~K.},
      author={Ombrosi, S.},
      author={Rivera-R\'{\i}os, I.~P.},
       title={On two weight estimates for iterated commutators},
        date={2021},
        ISSN={0022-1236},
     journal={J. Funct. Anal.},
      volume={281},
      number={8},
       pages={Paper No. 109153, 46},
         url={https://doi.org/10.1016/j.jfa.2021.109153},
}

\bib{Li}{article}{
      author={Li, K.},
       title={Multilinear commutators in the two-weight setting},
        date={2021},
     journal={Bull. Lond. Math. Soc.},
      eprint={https://arxiv.org/abs/2006.09071},
      status={to appear},
}

\bib{LMO}{article}{
      author={Li, K.},
      author={Martell, J.~M.},
      author={Ombrosi, S.},
       title={Extrapolation for multilinear {M}uckenhoupt classes and
  applications},
        date={2020},
        ISSN={0001-8708},
     journal={Adv. Math.},
      volume={373},
       pages={107286},
         url={https://doi.org/10.1016/j.aim.2020.107286},
}

\bib{LMMOV}{article}{
      author={Li, K.},
      author={Martell, J.M.},
      author={Martikainen, H:},
      author={Ombrosi, S.},
      author={Vuorinen, E.},
       title={End-point estimates, extrapolation for multilinear {M}uckenhoupt
  classes, and applications},
        date={2021},
        ISSN={0002-9947},
     journal={Trans. Amer. Math. Soc.},
      volume={374},
      number={1},
       pages={97\ndash 135},
         url={https://doi.org/10.1090/tran/8172},
      review={\MR{4188179}},
}

\bib{LMV:Bloom2}{article}{
      author={{Li}, K.},
      author={{Martikainen}, H.},
      author={{Vuorinen}, E.},
       title={{Bloom type upper bounds in the product BMO setting}},
        date={2019},
     journal={J. Geom. Anal.},
      volume={30},
       pages={3181\ndash 3203},
}

\bib{LMV:gen}{article}{
      author={Li, K.},
      author={Martikainen, H.},
      author={Vuorinen, E.},
       title={Genuinely multilinear weighted estimates for singular integrals
  in product spaces},
        date={2020},
      eprint={https://arxiv.org/abs/2011.04459},
      status={preprint},
}

\bib{LMV:Bloom}{article}{
      author={Li, K.},
      author={Martikainen, H.},
      author={Vuorinen, E.},
       title={{Bloom-Type Inequality for Bi-Parameter Singular Integrals:
  Efficient Proof and Iterated Commutators}},
        date={2021},
        ISSN={1073-7928},
     journal={Int. Math. Res. Not. IMRN},
      number={11},
       pages={8153\ndash 8187},
         url={https://doi.org/10.1093/imrn/rnz072},
}

\bib{LMS}{article}{
      author={Li, K.},
      author={Moen, K.},
      author={Sun, W.},
       title={The sharp weighted bound for multilinear maximal functions and
  {C}alder\'on--{Z}ygmund operators},
        date={2014},
     journal={J. Fourier Anal. Appl.},
      volume={20},
       pages={751\ndash 756},
}

\bib{Nieraeth}{article}{
      author={Nieraeth, B.},
       title={Quantitative estimates and extrapolation for multilinear weight
  classes},
        date={2019},
        ISSN={0025-5831},
     journal={Math. Ann.},
      volume={375},
      number={1-2},
       pages={453\ndash 507},
         url={https://doi.org/10.1007/s00208-019-01816-5},
}

\bib{Wi}{book}{
      author={Wilson, M.},
       title={Weighted {L}ittlewood-{P}aley theory and exponential-square
  integrability},
      series={Lecture Notes in Mathematics},
   publisher={Springer, Berlin},
        date={2008},
      volume={1924},
        ISBN={978-3-540-74582-2},
}

\bib{Wu}{article}{
      author={Wu, S.},
       title={A wavelet characterization for weighted {H}ardy spaces},
        date={1992},
        ISSN={0213-2230},
     journal={Rev. Mat. Iberoam.},
      volume={8},
      number={3},
       pages={329\ndash 349},
         url={https://doi.org/10.4171/RMI/127},
}

\end{biblist}
\end{bibdiv}

\end{document}